\documentclass{article}
\usepackage{graphicx}
\usepackage{amsmath}
\usepackage{hyperref}
\usepackage[all,rotate,2cell]{xy}
\usepackage{amsfonts}
\usepackage{amssymb}
\usepackage{url}
\usepackage{verbatim}

\UseTwocells

\newtheorem{theorem}{Theorem}

\newtheorem{corollary}[theorem]{Corollary}
\newtheorem{definition}[theorem]{Definition}

\newtheorem{lemma}[theorem]{Lemma}
\newtheorem{proposition}[theorem]{Proposition}
\newtheorem{remark}[theorem]{Remark}
\newenvironment{proof}[1][Proof]{\textbf{#1.} }{\ \rule{0.5em}{0.5em}}

\newcommand{\catname}[1]{\mathcal{#1}}
\newcommand{\catA}{\catname{A}}
\newcommand{\catB}{\catname{B}}
\newcommand{\catC}{\catname{C}}

\newcommand{\baseS}{\catname{S}}             
\newcommand{\Fin}{\mathsf{Fin}}              
\newcommand{\Sk}{\mathsf{Sk}}                
\newcommand{\thT}{\mathbb{T}}                
\newcommand{\thU}{\mathbb{U}}                
\newcommand{\thone}{\mathrm{1\!\!1}}         
\newcommand{\thob}{\mathbb{O}}               
\newcommand{\skext}{\subset}                 
\newcommand{\skeqext}{\Subset}               
\newcommand{\skhom}{\lessdot}              
\newcommand{\skhommap}{\gtrdot}               
\newcommand{\skdiag}{\sim}                   

\newcommand{\Mod}[2]{{#1\text{-}\mathbf{Mod}(#2)}}  
\newcommand{\Mods}[2]{{#1\text{-}\mathbf{Mod}_{s}(#2)}}  

\newcommand{\cod}{\mathop{\mathsf{cod}}}
\newcommand{\dom}{\mathop{\mathsf{dom}}}
\newcommand{\Id}{\mathop{\mathsf{Id}}}
\newcommand{\id}{\mathsf{id}}                
\newcommand{\cons}{\mathop{\mathsf{cons}}}   

\newcommand{\AU}{\mathbf{AU}}               
\newcommand{\AUpres}[1]{\AU\langle #1 \rangle} 
\newcommand{\Con}{\mathfrak{Con}}           
\newcommand{\pt}{\mathrm{pt}}               
\newcommand{\Loc}{\mathbf{Loc}}             
\newcommand{\Top}{\mathbf{Top}}             
\newcommand{\Cat}{\mathbf{Cat}}             
\newcommand{\CAT}{\mathfrak{CAT}}           

\newcommand{\List}{\mathop{\mathsf{List}}}  

\newcommand{\qeqobj}{\mathrm{obj}}          
\newcommand{\qeqarr}{\mathrm{arr}}          
\newcommand{\qeqdom}{\mathsf{d}}            
\newcommand{\qeqcod}{\mathsf{c}}            
\newcommand{\qeqid}{\mathsf{id}}            
\newcommand{\qeqcomp}{\circ}                
\newcommand{\qeqt}{1}                 
\newcommand{\qeqtfill}{\mathop{!}\nolimits^{\qeqt}}                 
\newcommand{\qeqprodproj}{\mathsf{p}}       
\newcommand{\qeqproj}[2]{\qeqprodproj_{#1,#2}}
\newcommand{\qeqpb}[2]{\mathsf{P}_{#1,#2}}             
\newcommand{\qeqpbfill}[4]{\left\langle #1, #2 \right\rangle_{#3,#4}} 
\newcommand{\qeqprodfill}[2]{\left\langle #1, #2 \right\rangle} 
\newcommand{\qeqeq}[2]{\mathsf{eq}_{#1,#2}}  
\newcommand{\qeqeqdom}[2]{\mathsf{E}_{#1,#2}}

\newcommand{\qeqi}{0}              
\newcommand{\qeqifill}{\mathop{!}\nolimits^{\qeqi}}                 
\newcommand{\qeqinj}[2]{\mathsf{q}_{#1,#2}} 
\newcommand{\qeqpo}[2]{\mathsf{Q}_{#1,#2}}             
\newcommand{\qeqpofill}[4]{\left[ #1, #2 \right]_{#3,#4}} 
\newcommand{\qequc}{\mathsf{uc}}            
\newcommand{\qeqpostab}[3]{\mathsf{stab}_{#1,#2}(#3)} 
\newcommand{\qeqex}{\mathsf{ex}}            

\newcommand{\qeqle}{\varepsilon}        
\newcommand{\qeqlcons}{\cons}        
\newcommand{\qeql}{\mathsf{List}}           
\newcommand{\qeqlrec}[3]{\mathsf{rec}^{#1}(#2,#3)} 

\newcommand{\skn}{\mathrm{G}^{0}}    
\newcommand{\sknid}{\mathrm{s}}       

\newcommand{\ske}{\mathrm{G}^{1}}    
\newcommand{\skface}{\mathrm{d}}     
\newcommand{\skedom}{\skface_0}     
\newcommand{\skecod}{\skface_1}     

\newcommand{\sktri}{\mathrm{G}^{2}}  
\newcommand{\sktril}{\skface_0}   
\newcommand{\sktrir}{\skface_2}   
\newcommand{\sktric}{\skface_1}   

\newcommand{\skut}{\mathrm{U}^{\qeqt}} 
\newcommand{\skutn}{\mathrm{t}}      

\newcommand{\skupb}{\mathrm{U}^{\mathrm{pb}}} 
\newcommand{\skupbtri}{\Gamma}       

\newcommand{\skui}{\mathrm{U}^{\qeqi}} 
\newcommand{\skuin}{\mathrm{i}}      

\newcommand{\skupo}{\mathrm{U}^{\mathrm{po}}} 
\newcommand{\skupotri}{\Gamma'}       

\newcommand{\skul}{\mathrm{U}^{\mathrm{list}}} 
\newcommand{\skulpb}{\Lambda_2}      
\newcommand{\skult}{\Lambda_0}       
\newcommand{\skule}{\mathrm{e}}      
\newcommand{\skulcons}{\mathrm{c}}   

\newcommand{\extdat}{\mathsf{Dat}}               
\newcommand{\extdatse}{\extdat_{\mathrm{s}\skext}}    
\newcommand{\extdatsee}{\extdat_{\mathrm{s}\skeqext}}  

\newcommand{\seq}[3]{
\xymatrix{
  {#1} \ar@{|-}[r]^-{#2}
  & {#3}
}
 }                                           
\newcommand{\converges}[1]{{#1\!\downarrow}} 

\begin{document}

\title{Sketches for arithmetic universes}
\author{Steven Vickers
        \\School of Computer Science, University of Birmingham
        \\ \texttt{s.j.vickers@cs.bham.ac.uk}}

\maketitle
\begin{abstract}
  A theory of sketches for arithmetic universes (AUs) is developed.

  A restricted notion of sketch, called here \emph{context},
  is defined with the property that every non-strict model is uniquely isomorphic to a
  strict model.
  This allows us to reconcile the syntactic, dealt with strictly using universal algebra,
  with the semantic, in which non-strict models must be considered.

  For any context $\thT$, a concrete construction is given of the AU $\AUpres{\thT}$
  freely generated by it.

  A 2-category $\Con$ of contexts is defined, with a full and faithful 2-functor to the 2-category
  of AUs and strict AU-functors,
  given by $\thT \mapsto \AUpres{\thT}$.
  It has finite pie limits,
  and also all pullbacks of a certain class of ``extension'' maps.
  Every object, morphism or 2-cell of $\Con$ is a finite structure.
\end{abstract}

\section{Introduction}\label{sec:Intro}
This paper arises out of a programme~\cite{TopCat} to use arithmetic universes (AUs)
to provide a predicative and base-free surrogate for Grothendieck toposes as generalized spaces
(and covering also point-free ungeneralized spaces such as locales or formal topologies).

Briefly, a generalized space is presented by a geometric theory $\thT$ that describes
-- as its models -- the points of the space,
and then the classifying topos $\baseS[\thT]$ is a presentation-independent representation
of the space.
In the case of a theory for an ungeneralized space, the topos is the category of sheaves.
In general, it embodies (as its internal logic) the ``geometric mathematics''
generated by a generic model of $\thT$.
In other words, it is the Grothendieck topos presented by $\thT$ as a system
of generators and relations.

Continuous maps (geometric morphisms) can be expressed as models of one theory
in the classifying topos of another
-- this is the universal property of ``classifying topos'' --
and so this also provides a logical account of continuity.
A map from $\thT_1$ to $\thT_2$ is defined by declaring,
``Let $M$ be a model of $\thT_1$,''
and then defining, in that context (in other words, in $\baseS[\thT_1]$,
with $M$ the generic model),
and within the constraints of geometricity, a model of $\thT_2$.
From this point of view one might say that continuity \emph{is} logical geometricity.
See~\cite{Vickers:ContIsGeom} or~\cite{LocTopSp} for a more detailed account of the ideas.

A significant problem in the approach is that the notions of Grothendieck topos
and classifying topos are parametrized by the base topos $\baseS$,
whose objects supply the infinities needed for the infinite disjunctions
needed in geometric logic,
and for the infinite coproducts needed in the category of sheaves
-- for example, to supply a natural numbers object.
Technically, Grothendieck toposes (with respect to $\baseS$) are then
elementary toposes equipped with bounded geometric morphisms to $\baseS$.

The aim of the AU programme is to develop a framework in which spaces, maps
and other constructions can be described in a way that does not depend on any
choice of base topos.
In this ``arithmetic'' logic, disjunctions would all be finite,
but some countable disjunctions could be dealt with by existential
quantification over infinite objects (such as $\mathbb{N}$) defined using the
list objects of AUs.
Thus those infinite disjunctions become an intrinsic part of the logic
-- albeit a logic with aspects of a type theory --
rather than being extrinsically defined by reference to a natural numbers
object in a base topos.

Now suppose a geometric theory $\thT$ can be expressed in this arithmetic way.
We write $\AUpres{\thT}$ for its classifying AU, which stands in for the base-dependent
classifying topos $\baseS[\thT]$.
An AU-functor%
\footnote{For the moment we ignore issues of strictness.}
$h\colon \AUpres{\thT_1} \to \AUpres{\thT_0}$
will, by composition, transform models of $\thT_0$ in any AU into models of $\thT_1$,
and is fruitfully thought of a point-free map between ``spaces of models'' of the two theories.
In particular, for any base topos $\baseS$ with nno,
$h$ will transform the generic model of $\thT_0$ in $\baseS[\thT_0]$ into a model of $\thT_1$
and so induce a geometric morphism from $\baseS[\thT_0]$ to $\baseS[\thT_1]$.
Thus a result expressed using AUs would provide a single statement of
a topos result valid over any base topos with nno.

It is already known that a range of results proved using geometric logic
can in fact be expressed in the setting of AUs.
\cite{ArithInd} develops some techniques for dealing with the fact that AUs
are not cartesian closed in general, nor even Heyting pretoposes.

This would be fully predicative, in that it does not at any point rely on
the impredicative theory of elementary toposes (with their power objects).
Instead of a predicative geometric theory of Grothendieck toposes,
parametrized by an impredicative base elementary topos,
we have a predicative arithmetic logic of AUs that is itself
internalizable in AUs,
and so depends on a predicative ambient logic.
(This internalizability aspect will be seen in, e.g., Section~\ref{sec:ConcreteAUT},
where we give a concrete construction of the AU presented by a context.)

In the present paper we propose a definition of arithmetic theory $\thT$
(our \emph{contexts}) and define a 2-category $\Con$ (Section~\ref{sec:Con})
that deals with the classifying AUs $\AUpres{\thT}$
in an entirely finitary way using presentations.

\begin{itemize}
\item
  The objects are (certain) finite presentations for AUs.
\item
  The collection of objects is rich enough to encompass practical mathematics
  including the real numbers.
\item
  The morphisms and 2-cells are such as to give a full and faithful 2-functor to AUs
  when the presentations are interpreted as the AUs that they present.
\end{itemize}

\emph{Presentations:}
In principle, the quasiequational theories of~\cite{PHLCC} provide a means of
presenting AUs by generators and relations.
However, for various reasons we find it more convenient to use a technique
based on sketches (Section~\ref{sec:AUSk}).
Our ``contexts'' (Section~\ref{sec:Extensions}) are then a restricted form of sketches,
built up by finitely many steps of adjoining objects, morphisms, commutativities,
and ``universals'' (for limit cones, colimit cocones, and list objects).

The main difference from quasiequational presentations is that the contexts do not allow
the possibility of expressing equality between objects,
except when they are either declared
as the same node or constructed by identical universal constructions from equal data.

This restriction is also relevant when it comes to \emph{Strictness:}
The technology of universal algebra relies on the universal constructions such as pullbacks
being interpreted strictly,
since in the algebra they appear as expressions.
As part of this, when one considers AU-functors between the AUs presented by presentations,
it is only the strict AU-functors that can be described exactly in terms of the presentations.

On the other hand, non-strict AU-functors will be important,
particularly in topos applications.
Although every elementary topos with nno is an AU,
and every inverse image functor part of a geometric morphism is an AU-functor,
it is highly unlikely to be strict.

Models of an AU sketch can be interpreted in the non-strict way that is usual for sketches,
but can also be interpreted strictly.
Then an advantage of our \emph{contexts} is that each non-strict model is uniquely isomorphic
to a strict model.
(The restrictions on our ability to express equalities between pairs of nodes are important here.)
Hence it is straightforward to apply the strict theory to non-strict models.

\emph{Full faithfulness:}
A principal goal (Theorem~\ref{thm:AUpres}) is that arbitrary strict AU-functors
between presented AUs should be expressible up to equality in terms of the presenting contexts.
Our initial notion of morphism between contexts is that of sketch homomorphism,
but this is entirely syntax-bound and insufficiently general.
It maps nodes to nodes, edges to edges, commutativities to commutativities, etc.
We need two technical ingredients to get beyond this.

\emph{Object equalities} (Section~\ref{sec:ObjEq})
deal with the fact that,
although our contexts do not allow us to express arbitrary equalities between objects,
implied equalities can arise when identical constructions are applied to equal data.
An object equality between objects is a fillin morphism that arises in that kind of way.
Note that this is much stronger than simply having an isomorphism.
We extend the phrase ``object equality'' to apply more generally to
homomorphisms of models in which every carrier morphism is an object equality.

\emph{Equivalence extensions} (Section~\ref{sec:EquivExt})
accommodate our need to map elements of one context not just to elements explicitly
in another (which is what a context homomorphism does),
but also to derived elements.
An equivalence extension of a context adjoins elements that are uniquely determined by
elements of the original,
so that the presented AUs are isomorphic.
This is essentially the idea of ``schema entailment'' as set out in~\cite{GeoZ}.

Our category $\Con$ (Section~\ref{sec:Con}),
which maps fully and faithfully to AUs and strict AU-functors,
is then made by turning object equalities to equalities
and making equivalence extensions invertible.

\emph{Note on notation:} Our default order of composition of morphisms is \emph{diagrammatic}.
For applicational order we shall always use ``$\circ$''.
For diagrammatic order we shall occasionally show this explicitly using ``;''.

\section{Arithmetic universes}\label{sec:AUs}
We follow~\cite{Maietti:AritJ,ArithInd} in defining Joyal's arithmetic universes (AUs)
to be \emph{list arithmetic pretoposes}.

More explicitly,
as a pretopos an AU $\catA$ is a category equipped with finite limits,
stable finite disjoint coproducts and stable effective quotients of equivalence relations.
(For more detailed discussion, see, e.g., \cite[A1.4.8]{Elephant1}.)

In addition, it has, for each object $A$,
a \emph{parametrized list object} $\qeql(A)$.
It is equipped with morphisms
\[
  \xymatrix{
    {1}
      \ar@{->}[r]^-{\varepsilon}
    & {\qeql(A)}
      \ar@{<-}[r]^-{\cons}
    & {A \times \qeql(A)}
  }
\]
(where $\cons(a,x)=a:x$ is the list $x$ with $a$ appended at the front)
and whenever we have the solid part of the following diagram, there is a unique
fillin of the dotted parts to make a commutative diagram.
\begin{equation}\label{eq:listunichar}
  \xymatrix@C=2cm{
    & {\qeql(A)\times B}
      \ar@{.>}[dd]^{\qeqlrec{A}{y}{g}}
    & {(A\times \qeql(A)) \times B}
      \ar@{->}[l]_{\cons \times B}
      \ar@{->}[d]^{\cong}
    \\
    & & {A\times (\qeql(A)\times B)}
      \ar@{.>}[d]^{A\times\qeqlrec{A}{y}{g}}
    \\
    {B}
      \ar@{->}[uur]^{\langle \varepsilon,B \rangle}
      \ar@{->}[r]_{y}
    & {Y}
      \ar@{<-}[r]_{g}
    & {A\times Y}
  }
\end{equation}
In other words, this recursively defines $r = \qeqlrec{A}{y}{g}$ by
  \begin{gather*}
    r([],b) = y(b) \\
    r(a:x,b) = g(a,r(x,b))
  \end{gather*}


Note that the use of $B$ rather than $1$ corresponds to this being a \emph{parameterized}
list object -- that is to say, it makes $\qeql(A)\times B$ a list object in the slice over $B$.

\begin{remark}\label{rem:listf}
  For future reference, we note the functoriality of $\qeql$:
  If $f\colon A_1 \to A_2$, then there is a unique
  $\qeql(f)\colon\qeql(A_1)\to\qeql(A_2)$ making the following diagram commute.
  \[
    \xymatrix@C=2cm{
      & {\qeql(A_1)}
        \ar@{<-}[r]^{\cons_1}
        \ar@{.>}[d]^{\qeql(f)}
      & {A_1 \times \qeql(A_1)}
        \ar@{.>}[d]^{f\times\qeql(f)}
      \\
      {1}
        \ar@{->}[ur]^{\varepsilon_1}
        \ar@{->}[r]_{\varepsilon_2}
      & {\qeql(A_2)}
        \ar@{<-}[r]_{\cons_2}
      & {A_2\times\qeql(A_2)}
    }
  \]
  To see this, consider the action of $A_1$ on $\qeql(A_2)$ by
  \[
    \xymatrix@1@C=2cm{
      {\qeql(A_2)}
      & {A_2\times\qeql(A_2)}
        \ar@{->}[l]_{\cons_2}
      & {A_1\times \qeql(A_2)}
        \ar@{->}[l]_{f\times \qeql(A_2)}
    }
    \text{.}
  \]
\end{remark}

We assume the AU structure specifies canonical choices
of those colimits, limits and list objects.
This enables an approach using the universal algebra of cartesian theories,
with (partial) algebraic operators for the canonical choices.

We shall use the quasiequational form of cartesian theories~\cite{PHLCC}.
Our cartesian theory of AUs will use primitive operators
as suggested by the following proposition,
although that particular choice of primitives is not critical.
Doubtless there are more efficient characterizations,
and the techniques in the remainder of the present paper
are intended to be equally applicable for other choices.

\begin{proposition}
  A category $\catA$ is an arithmetic universe iff the following hold.
  \begin{enumerate}
  \item
    $\catA$ has a terminal object and pullbacks (hence all finite limits).
  \item
    $\catA$ has an initial object and pushouts (hence all finite colimits),
    and they are stable under pullback.
  \item
    Balance (unique choice):
    if a morphism is both mono and epi, then it is iso.
  \item
    Exactness: any equivalence relation is effective
    (it is the kernel pair of its own coequalizer).
  \item
    $\catA$ has parameterized list objects.
  \end{enumerate}
\end{proposition}
\begin{proof}
  $\Rightarrow$:
  (1), (3) and (4) are properties of any pretopos,
  as is the existence of stable finite coproducts.
  (5) is a postulate for AUs.

  Hence it remains to show the existence of stable coequalizers for all pairs $X\rightrightarrows Y$.
  First, because, as pretopos, $\catA$ is regular, we can take the image $R$ in $Y\times Y$,
  a relation on $Y$.
  Next, in a pretopos we can find the reflexive-symmetric closure of $R$.
  Next, in an AU we can find the free category over any directed graph,
  and in particular we can find the transitive closure of any relation.
  We end up with the equivalence relation generated by $R$,
  and at each step, we keep the same set of morphisms from $Y$ that compose equally with
  the two morphisms from $X$ or $R$.
  Thus the coequalizer of the equivalence relation,
  existing because of exactness of $\catA$ as pretopos,
  also serves as a coequalizer for $X\rightrightarrows Y$.

  Stability follows from the stability, in a pretopos, of image factorization
  and of coequalizers of equivalence relations.

  $\Leftarrow$:
  Two properties of pretoposes remain to be proved.
  First, for binary coproduct, the injections are monic and disjoint.
  Second, any epi is the coequalizer of its kernel pair.

  Consider a coproduct cocone (bottom row here) pulled back along one of the injections.
  The two squares are pullbacks, $\Delta$ is diagonal.
  \[
    \xymatrix{
      {K}
        \ar@{->}[r]^{p_2}
        \ar@{->}[d]_{p_1}
      & {X}
        \ar@{->}@/^1pc/[l]^{\Delta}
        \ar@{->}[d]^{i_1}
      & {L}
        \ar@{->}[l]_{q_2}
        \ar@{->}[d]^{q_1}
      \\
      {X}
        \ar@{->}[r]_{i_1}
      & {X+Y}
      & {Y}
        \ar@{->}[l]^{i_2}
    }
  \]
  By stability the top row is a coproduct cocone,
  and so we can define a copairing morphism $f = [\Id_{K}, q_2 \Delta] \colon X \to K$
  and calculate that $f = p_2^{-1}$.
  Since the kernel pair projection $p_2$ is an isomorphism,
  it follows that $i_1$ is monic.

  We can now replace $K$ and the projections by $X$ and identity morphisms,
  and the coproduct property of the top row can be rephrased as follows:
  every triangle
  $ \xymatrix{
    {L}
      \ar@{->}[r]_{q_2}
      \ar@{->}@/^1pc/[rr]^{}
    & {X}
      \ar@{->}[r]
    & {}
  }$
  commutes.

  Now consider $! \colon 0 \to L$.
  By stability of the initial object, we see that $!$ is mono.
  It is also epi.
  For suppose we have two morphisms $f_1,f_2 \colon L \to Z$.
  Consider the following diagram, where $j_1$ and $j_2$ are coproduct injections.
  \[
    \xymatrix{
      {L}
        \ar@{->}[r]^{q_2}
        \ar@{->}@<1ex>[d]^{f_2}
        \ar@{->}@<-1ex>[d]_{f_1}
      & {X}
        \ar@{->}[d]^{j_2}
      \\
      {Z}
        \ar@{->}[r]^{j_1}
      & {Z+X}
    }
  \]
  Both squares must commute, and we already know that $j_1$ is monic,
  so $f_1=f_2$.
  By balance, it follows that $L\cong 0$.

  It remains to show that any epi $e\colon X \to Y$ is the coequalizer of its kernel pair.
  In fact we show something slightly more general, without assuming $e$ is epi.
  Let $K_2$ be its kernel pair, with projections $p_1$ and $p_2$,
  and let $e'\colon X \to Y'$ be their coequalizer,
  with factorization $e = e' e''$.
  Then we show that $e''$ is mono.
  (If $e$ is epi then so too is $e''$, so $e''$ is an isomorphism by balance.)

  In the following diagram, where the bottom row is pulled back along $e$,
  we see that the top row is a split fork and hence a coequalizer.
  \[
    \xymatrix{
      {K_3}
        \ar@{->}@<1ex>[r]^{p_{13}}
        \ar@{->}[r]_{p_{23}}
        \ar@{->}[d]_{p_{12}}
      & {K_2}
        \ar@{->}@/^1pc/[l]^{\Delta_{122}}
        \ar@{->}[r]^{p_2}
        \ar@{->}[d]^{p_1}
      & {X}
        \ar@{->}@/^1pc/[l]^{\Delta}
        \ar@{->}[d]^{e}
      \\
      {K_2}
        \ar@{->}@<0.5ex>[r]^{p_{1}}
        \ar@{->}@<-0.5ex>[r]_{p_{2}}
      & {X}
        \ar@{->}[r]^{e}
      & {Y}
    }
  \]
  Now consider pulling back the factorization $e' e''$:
  \[
    \xymatrix{
      {K_2}
        \ar@{->}[r]^{f'}
        \ar@{->}[d]_{p_1}
      & {X'}
        \ar@{->}[r]^{f''}
        \ar@{->}[d]^{g}
      & {X}
        \ar@{->}@/^1pc/[ll]^(0.7){\Delta}
        \ar@{->}[d]^{e}
      \\
      {X}
        \ar@{->}[r]_{e'}
      & {Y'}
        \ar@{->}[r]_{e''}
      & {Y}
    }
  \]
  By stability, we see that $f'$ is a coequalizer of $p_{13}$ and $p_{23}$,
  and by comparing with the split fork we find $\Delta f' = f''^{-1}$.
  We deduce that we can replace $X'$ by $X$, and have a pullback square
  \[
    \xymatrix{
      {X}
        \ar@{=}[r]
        \ar@{->}[d]_{e'}
      & {X}
        \ar@{->}[d]^{e}
      \\
      {Y'}
        \ar@{->}[r]_{e''}
      & {Y}
    }
  \]
  Now we can use the above pullback square, turned on its side,
  to pull back the factorization $e' e''$ along $e''$.
  \[
    \xymatrix{
      & {L}
        \ar@{->}[dr]^{p_2}
        \ar@{->}[dd]^(0.7){p_1}
      \\
      {X}
        \ar@{->}[ur]^{\langle e',e' \rangle}
        \ar@{->}[rr]_(0.7){e'}
        \ar@{=}[dd]
      && {Y'}
        \ar@{->}[dd]^{e''}
      \\
      & {Y'}
        \ar@{->}[dr]^{e''}
      \\
      {X}
        \ar@{->}[ur]^{e'}
        \ar@{->}[rr]_{e}
      && {Y}
    }
  \]
  By stability (for $e'$ as coequalizer) we see that $p_2$ is an isomorphism
  and so $e''$ is monic.
\end{proof}

It follows that the theory of AUs is quasiequational as in~\cite{PHLCC}.

\begin{definition}\label{def:AUQuasieq}
  We present the quasiequational theory of AUs as follows.
  Some of the operators and axioms are already set out explicitly in~\cite{PHLCC,ArithInd},
  and we refer back to them for some of the details.
  \begin{itemize}
  \item
    (See~\cite[Example~4]{PHLCC}.)
    The ingredients of the theory of categories:
    sorts $\qeqobj,\qeqarr$,
    total operators $\qeqdom,\qeqcod\colon\qeqarr\to\qeqobj$ (domain and codomain) and
    $\qeqid \colon \qeqobj \to \qeqarr$ (identity morphisms),
    and partial $\qeqcomp\colon\qeqarr^2 \to \qeqarr$
    (composition, as binary operator, in applicational order).
  \item
    (See~\cite[Section~6]{PHLCC}.)
    Ingredients for terminal objects:
    a constant $\qeqt\colon\qeqobj$
    and operator $\qeqtfill_{\cdot}\colon\qeqobj\to\qeqarr$ (unique morphism to terminal) with axioms
    \begin{gather*}
      \seq{\top}{X}{\qeqdom(\qeqtfill_{X})=X \wedge \qeqcod(\qeqtfill_{X})=\qeqt}\\
      \seq{\qeqcod(u)=\qeqt}{u}{u = \qeqtfill_{\qeqdom(u)}}
      \text{ (for uniqueness).}
    \end{gather*}
  \item
    (See~\cite[Section~6.1]{PHLCC}.)
    Ingredients for pullbacks:

    First, operators
    $\qeqproj{\cdot}{\cdot}^1, \qeqproj{\cdot}{\cdot}^2 \colon \qeqarr^2 \to \qeqarr$
    for pullback projections.
    If $u_1$ and $u_2$ have a common codomain,
    then $\qeqproj{u_1}{u_2}^1$ and $\qeqproj{u_1}{u_2}^2$ are the two projections from the pullback.
    We also write $\qeqproj{u_1}{u_2}$ for $u_1 \qeqcomp \qeqproj{u_1}{u_2}^{1}$,
    the diagonal of the pullback square,
    and $\qeqpb{u_1}{u_2}$ for $\qeqdom(\qeqproj{u_1}{u_2})$,
    the pullback object itself.

    Next, a pairing operator
    $\qeqpbfill{\cdot}{\cdot}{\cdot}{\cdot} \colon \qeqarr^4 \to \qeqarr$,
    with $\qeqpbfill{v_1}{v_2}{u_1}{u_2}$
    the fillin to the pullback of $u_1$ and $u_2$ for a cone $(v_1, v_2)$.
    It is defined iff the four arrows make a commutative square in the obvious way,
    and it has the expected domain and codomain and commutativities.

    For uniqueness of fillins,
    \[
      \seq{\qeqcod(w) = \qeqpb{u_1}{u_2}}{w,u_1,u_2}
          {w = \qeqpbfill
                  {\qeqproj{u_1}{u_2}^1 \qeqcomp w}
                  {\qeqproj{u_1}{u_2}^2 \qeqcomp w}
                  {u_1}{u_2}
          }
      \text{.}
    \]
  \item
    We shall also use some derived notation in a self-explanatory way for products
    $X\times Y = \qeqpb{\qeqtfill_X}{\qeqtfill_Y}$.
    For example, the projections are
    $\qeqproj{X}{Y}^i = \qeqproj{\qeqtfill_X}{\qeqtfill_Y}^i$,
    and the fillins require no subscripts.

    Also, we shall write $\qeqeq{u_1}{u_2}\colon\qeqeqdom{u_1}{u_2}\to X$
    for the equalizer of $u_1,u_2\colon X \to Y$, defined in a canonical way.
    Specifically,
    \[
      \qeqeq{u_1}{u_2} \triangleq
          \qeqproj{\qeqprodfill{\qeqid(X)}{u_1}}{\qeqprodfill{\qeqid(X)}{u_2}}^1
    \]
    (The two projections are equal.)
  \item
    Ingredients for initial objects and pushouts.
    They are dual to those for terminal objects and pullbacks.
    (We can also express coproducts and coequalizer,
    by dualizing the treatment for products and equalizers.)

    For initial objects we have a constant $\qeqi\colon\qeqobj$,
    an operator $\qeqifill_{\cdot}\colon\qeqobj\to\qeqarr$,
    and a conditional equation that if $\qeqdom(u)=\qeqi$,
    then $u = \qeqifill_{\qeqcod(u)}$.

    Operators
    $\qeqinj{\cdot}{\cdot}^1, \qeqinj{\cdot}{\cdot}^2 \colon \qeqarr^2 \to \qeqarr$
    are for pushout injections.
    If $u_1$ and $u_2$ have a common domain,
    then $\qeqinj{u_1}{u_2}^1$ and $\qeqinj{u_1}{u_2}^2$ are the two injections to the pushout.
    We also write $\qeqinj{u_1}{u_2}$ for $\qeqinj{u_1}{u_2}^{1} \qeqcomp u_1$,
    the diagonal of the pushout square,
    and $\qeqpo{u_1}{u_2}$ for $\qeqcod(\qeqinj{u_1}{u_2})$,
    the pushout object itself.

    Next, a copairing operator
    $\qeqpofill{\cdot}{\cdot}{\cdot}{\cdot} \colon \qeqarr^4 \to \qeqarr$,
    with $\qeqpofill{v_1}{v_2}{u_1}{u_2}$
    the fillin from the pushout of $u_1$ and $u_2$ for a cocone $(v_1, v_2)$.

    For uniqueness of fillins,
    \[
      \seq{\qeqdom(w)=\qeqpo{u_1}{u_2}}
        {w,u_1,u_2}
        {w = \qeqpofill
          {w \qeqcomp \qeqinj{u_1}{u_2}^1}
          {w \qeqcomp \qeqinj{u_1}{u_2}^2}
          {u_1}{u_2}
        }
      \text{.}
    \]
  \item
    Ingredients for stability of colimits under pullback.

    For stability of the initial object,
    it suffices to say that any morphism with $\qeqi$ for codomain is an isomorphism:
    \[
      \seq{\qeqcod(u) = \qeqi}
        {u}
        {\qeqifill_{\qeqdom(u)}\qeqcomp u = \qeqid(\qeqdom(u))}
      \text{.}
    \]

    For stability of pushouts,
    we have an operator
    $\qeqpostab{\cdot}{\cdot}{\cdot}\colon \qeqarr^3 \to \qeqarr$,
    with $\qeqpostab{u_1}{u_2}{w}$ defined iff $\qeqcod(w) = \qeqpo{u_1}{u_2}$.
    To express its equations, we define notation as shown in this diagram.
    Here the base diamond is a pushout, and it is pulled back along $w$.
    The inner top diamond is also a pushout, with fillin $e$,
    and the equations for the operator, when it is defined,
    are those required to say that
    $\qeqpostab{u_1}{u_2}{w} = e^{-1}$.
    
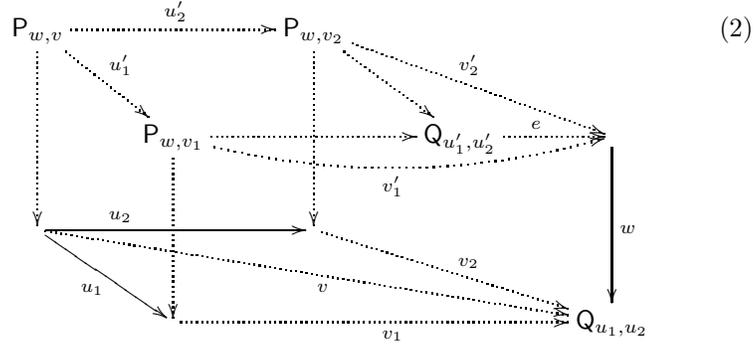
\begin{equation}\label{eq:poStab}
      \xymatrix{
        {\qeqpb{w}{v}}
          \ar@{.>}[rr]^{u'_2}
          \ar@{.>}[dr]^{u'_1}
          \ar@{.>}[dd]
        && {\qeqpb{w}{v_2}}
          \ar@{.>}[drr]^{v'_2}
          \ar@{.>}[dr]
          \ar@{.>}[dd]
        \\
        & {\qeqpb{w}{v_1}}
          \ar@{.>}[rr]
          \ar@{.>}@/_1pc/[rrr]_{v'_1}
          \ar@{.>}[dd]
        && {\qeqpo{u'_1}{u'_2}}
          \ar@{.>}[r]^{e}
        & {}
          \ar@{->}[dd]^{w}
        \\
        {}
          \ar@{->}[rr]^(.3){u_2}
          \ar@{.>}[drrrr]_{v}
          \ar@{->}[dr]_{u_1}
        && {}
          \ar@{.>}[drr]^{v_2}
        \\
        & {}
          \ar@{.>}[rrr]_{v_1}
        &&& {\qeqpo{u_1}{u_2}}
      }
    \end{equation}
    \[
      \begin{array}[t]{ll}
        v      & = \qeqinj{u_1}{u_2} \\
        v_{i}  & = \qeqinj{u_1}{u_2}^{i} \\
        v'_{i} & = \qeqproj{w}{v_{i}}^{1} \\
        u'_{i} & = \qeqpbfill{\qeqproj{w}{v}^1}{u_{i} \qeqcomp \qeqproj{w}{v}^2}{w}{v_{i}} \\
        e      & = \qeqpofill{v'_1}{v'_2}{u'_1}{u'_2}
      \end{array}
    \]
\item
    Ingredients for balance (unique choice).

    We have an operator
    $\qequc \colon \qeqarr \to \qeqarr$,
    with $\qequc(u)$ defined if
    $\qeqproj{u}{u}^1 = \qeqproj{u}{u}^2$
    and $\qeqinj{u}{u}^1 = \qeqinj{u}{u}^2$
    (i.e. $u$ is monic and epi).
    When it is defined we have
    $\qequc(u) = u^{-1}$.
  \item
    Ingredients for exactness.

    We have an operator $\qeqex \colon \qeqarr^5 \to \qeqarr$,
    with $\qeqex(\pi_1,\pi_2,r,s,t)$ defined if $\langle\pi_1,\pi_2\rangle$ describes
    a binary relation,
    with $r,s,t$ expressing reflexivity, symmetry and transitivity.

    \begin{equation}\label{eq:exactness}
      \xymatrix@C=2cm{
        & & {K}
          \ar@{->}[d]^{\qeqproj{\gamma}{\gamma}^{i} \, (i=1,2)}
        \\
        {X_2}
          \ar@{->}@<0.5ex>[r]^{\qeqproj{\pi_2}{\pi_1}^{i} \, (i=1,2)}
          \ar@{->}@<-0.5ex>[r]_{t}
        & {X_1}
          \ar@{->}@<0.5ex>[r]^{\pi_i \, (i=1,2)}
          \ar@{->}[ur]^{e}
          \ar@{->}[dr]_{\pi = \qeqprodfill{\pi_1}{\pi_2}}
          \ar@{->}@(ul,ur)[]^{s}
        & {X_0}
          \ar@{->}@<0.5ex>[l]^{r}
          \ar@{->}[r]_{\gamma}
        & {X}
        \\
        & & {X_0 \times X_0}
          \ar@{->}[u]_{\qeqproj{X_0}{X_0}^{i} \, (i=1,2)}
      }
    \end{equation}
    We require that $\pi$ in monic; that $X_2 = \qeqpb{\pi_2}{\pi_1}$;
    that $r,s,t$ compose correctly with $\pi_1$ and $\pi_2$;
    that $\gamma$ is the canonical coequalizer of $\pi_1$ and $\pi_2$;
    that $K$ is the kernel pair of $\gamma$;
    and that $e$ is the fillin.
    Our characterizing equations for $\qeqex$ are to say
    \[
      \qeqex(\pi_1,\pi_2,r,s,t) = e^{-1}\text{.}
    \]
  \item
    Ingredients for list objects.

    We have total operators
    $\qeqle,\qeqlcons \colon \qeqobj \to \qeqarr$
    for the principal structure,
    and we also write $\qeql(A)$ for $\qeqcod(\qeqle(A))$.

    For the fillins we have a partial operator
    $\qeqlrec{\cdot}{\cdot}{\cdot}\colon \qeqobj\times\qeqarr^2 \to \qeqarr$.

    Let us write, temporarily, the following. (See diagram~\eqref{eq:listunichar}.)
    \[
      \begin{split}
        \phi^{A}(y,g) & \triangleq
          \qeqcod(y)=\qeqcod(g)
          \wedge \qeqdom(g) = A \times \qeqcod(g) \\
        \psi^{A}_{y,g}(r) & \triangleq
          y = r \qeqcomp \qeqprodfill{B}{\varepsilon} \\
          & \wedge g \qeqcomp (r\times A) \qeqcomp \mathop{\cong}
                 = r \qeqcomp (B\times\qeqlcons(A)) \\
        \text{where} \\
        B & \triangleq \qeqdom(y) \\
        \qeqprodfill{B}{\varepsilon} & \triangleq
          \qeqprodfill{\qeqid(B)}{\qeqle(A) \qeqcomp \qeqtfill_B} \\
        \mathop{\cong} & \triangleq
          \qeqprodfill{\qeqproj{A}{\qeql A}^1 \qeqcomp \qeqproj{A\times\qeql A}{B}^1}
            {\qeqprodfill{\qeqproj{A}{\qeql A}^2 \qeqcomp \qeqproj{A\times\qeql A}{B}^1}
              {\qeqproj{A\times\qeql A}{B}^2}}
      \end{split}
    \]
    Here $\phi$ expresses the domain of definition of the fillin $\qeqlrec{A}{y}{g}$,
    and $\psi$ is the condition (on $r$) that it needs to satisfy. The axioms are now --
    \[
      \begin{split}
        \seq{\top}{A}{\phi^{A}(\qeqle(A), \qeqlcons(A))
            \wedge \qeqdom(\qeqle(A)) = 1} \\
        \seq{\phi^{A}(y,g)}{A,y,g}{\psi^{A}_{y,g}(\qeqlrec{A}{y}{g})} \\
        \seq{\converges{\qeqlrec{A}{y}{g}}}{A,y,g}{\phi^{A}(y,g)}    \\
        \seq{\psi^{A}_{y,g}(r)}{A,y,g,r}{r = \qeqlrec{A}{y}{g}}
      \end{split}
    \]
  \end{itemize}
\end{definition}

\begin{definition}\label{def:AUfunctor}
  A \emph{strict AU-functor} from one AU to another is a homomorphism for the quasiequational
  theory of AUs.
  In other words, it is a functor that preserves terminals, pullbacks, intials, pushouts
  and list objects \emph{strictly.}

  An \emph{AU-functor} is a functor that preserves those constructions
  (and hence also all finite limits and finite colimits) up to isomorphism.
\end{definition}

In AUs we have a general ability to construct free algebras.
For theories given by finite product (FP) sketches
this is described in some detail in~\cite{Maietti:lexsk}.
That paper also alludes to the ability to generalize to finite limit (FL) sketches,
in other words to cartesian theories.
\cite{PHLCC} gives a general account of the cartesian construction, and it is valid in AUs.

\section{AU-sketches}\label{sec:AUSk}
We shall be interested in generators and relations for AUs,
but we shall generally not express them directly using the quasiequational algebra.
Instead, we borrow the ideas of sketches.

In their most general form (in this section), they are equivalent in expressive power to the
quasiequational algebra.
In one direction we make this explicit by giving the equations that correspond to
ingredients of a sketch.
The other direction is less clear, but comes down to the question of how to express the operators
in the quasiequational theory of AUs.
The operators for pullbacks and their projections,
and analogous operators for other universal constructions,
can be captured using the ``universals'' in a sketch.
The operators for fillins, being the unique solutions to certain equational constraints on edges,
can be captured with edges constrained by suitable commutativities.

Our main reason for using the sketches is that they give us better control of the important
issue of strictness of models (Section~\ref{sec:Models}).
In Section~\ref{sec:Extensions} we shall restrict our attentions from general sketches
to ``contexts'', finite sketches for which we have good coherence properties for strictness.

\begin{definition}\label{def:AUSk}
  An \emph{AU-sketch} (or just \emph{sketch}) is a structure with sorts
  and operations as shown in this diagram.
  
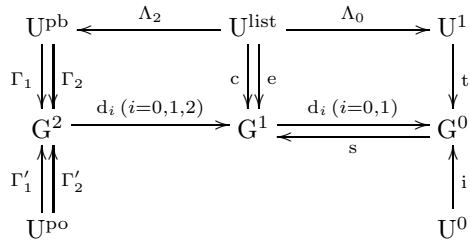
\[
    \xymatrix{
      {\skupb}
        \ar@{->}@<0.5ex>[d]^{\skupbtri_2}
        \ar@{->}@<-0.5ex>[d]_{\skupbtri_1}
      && {\skul}
        \ar@{->}[ll]_{\skulpb}
        \ar@{->}[rr]^{\skult}
        \ar@{->}@<0.5ex>[d]^{\skule}
        \ar@{->}@<-0.5ex>[d]_{\skulcons}
      && {\skut}
        \ar@{->}[d]^{\skutn}
      \\
      {\sktri}
        \ar@{->}@<0.5ex>[rr]^{\skface_{i}\,(i=0,1,2)}
      && {\ske}
        \ar@{->}@<0.5ex>[rr]^{\skface_{i}\,(i=0,1)}
      && {\skn}
        \ar@{->}@<0.5ex>[ll]^{\sknid}
      \\
      {\skupo}
        \ar@{->}@<0.5ex>[u]^{\skupotri_1}
        \ar@{->}@<-0.5ex>[u]_{\skupotri_2}
      &&&& {\skui}
        \ar@{->}[u]_{\skuin}
    }
  \]
  They are required to satisfy the following equations:
  \begin{align*}
    & \sknid\skedom = \sknid\skecod = \Id  \\
    & \sktril\skecod = \sktrir\skedom
       \quad \sktril\skedom = \sktric\skedom
       \quad \sktrir\skecod = \sktric\skecod  \\
    & \skupbtri_1\sktric =\skupbtri_2\sktric  \\
    & \skupotri_1\sktric =\skupotri_2\sktric  \\
    & \skulpb\skupbtri_1\sktric\skecod =\skult\skutn \\
    & \skule\skedom = \skult\skutn \quad \skulcons\skedom = \skulpb\skupbtri_1\sktril\skedom \\
    & \skule\skecod = \skulcons\skecod = \skulpb\skupbtri_1\sktril\skecod
  \end{align*}

  If $\thT_1$ and $\thT_2$ are sketches,
  then a \emph{homomorphism} of sketches from $\thT_1$ to $\thT_2$,
  written $f\colon\thT_1\skhom\thT_2$,
  is defined in the obvious way --
  a family of carrier functions, one for each sort, preserving the operators.

  However, we shall consider two sketch homomorphisms to be \emph{equal}
  if they agree merely on $\skn$ and $\ske$.

  We write $\mathbf{Sk}_{\skhom}$ for the category of sketches and sketch homomorphisms.
\end{definition}

The structures are a formalization of the sketches well known from e.g. \cite{BarrWells:TTT},
but adapted for AUs.
We shall describe the parts in more detail below,
but as a preliminary let us introduce some language that indicates the connection.
The elements of $\skn$, $\ske$ and $\sktri$ are referred to as
\emph{nodes, edges} and \emph{commutativities}.

The elements of the other sorts are \emph{universals,}
and specify universal properties of various kinds for their \emph{subjects}.
For example, an element of $\skupb$ is a \emph{pullback universal}
and corresponds to a cone in a finite limit sketch.
Its subjects are the pullback node and the three projection edges of the pullback cone.
Similarly, an element of $\skul$ is a \emph{list universal}.
Its subjects are the list object and the two structure maps, for $\qeqle$ and $\qeqlcons$.
It will also have indirect subjects,
since it needs terminal and pullback universals to express the domains of the structure maps.

Any sketch can be used as a system of generators
(the nodes and edges)
and relations to present an AU.
We shall list these implied relations in the general description below.
Note that in each case the equations constraining sketches ensure that
all the terms used in the relations are defined.

$\skn,\sknid, \ske,\skedom,\skecod$ form the graph
(which we take to be reflexive) of nodes and edges,
declaring some objects and arrows and specifying their identities, domains and codomains.
The elements of $\skn$ and $\ske$ are taken as generators of sorts $\qeqobj$ and $\qeqarr$.
The implied relations are --
\[
  \qeqid(X) = \sknid(X) \quad \qeqdom(u) = \skedom(u) \quad \qeqcod(u) = \skecod(u)
\]

$\sktri$, with $\sktril$, $\sktrir$ and $\sktric$, comprises the commutativities, stipulating commutative triangles
$\xymatrix@1{
  {}
    \ar@{->}[r]_{\sktril}
    \ar@{->}@/^1pc/[rr]_{\bullet}^{\sktric}
  & {}
    \ar@{->}[r]_{\sktrir}
  & {}
}$.
Given a triangle of edges
$\xymatrix@1{
  {X}
    \ar@{->}[r]_{u}
    \ar@{->}@/^1pc/[rr]^{w}
  & {Y}
    \ar@{->}[r]_{v}
  & {Z}
}$,
we shall write $uv \skdiag_{XYZ} w$ for the existence of a commutativity with that triangle.
(Note the \emph{diagrammatic order}.)
We shall also write $u\skdiag_{XY}u'$ for a \emph{unary commutativity},
meaning a commutativity $\sknid(X) u\skdiag_{XXY}u'$.
We shall omit the node subscripts where convenient.

Equationally, each commutativity $\omega$ corresponds to a relation
\[
  \sktrir(\omega)\qeqcomp\sktril(\omega) = \sktric(\omega)
  \text{.}
\]

$\skut$ and $\skupb$, using $\skutn,\skupbtri^1,\skupbtri^2$,
are universals for finite limits, here terminal objects or pullbacks.
For each pullback universal (in $\skupb$) we describe the cone by two commutative triangles
($\skupbtri^1,\skupbtri^2$),
the two halves of the pullback square.
For universals $\omega\in\skut$ or $\omega\in\skupb$, the implied relations are --
\[
  \skutn(\omega) = \qeqt
\]
\[
  \sktril(\skupbtri^{\lambda}(\omega)) =
    \qeqproj{\sktrir(\skupbtri^1(\omega))}{\sktrir(\skupbtri^2(\omega))}^{\lambda}
  \quad
  (\lambda = 1,2)
\]

$\skui,\skuin,\skupo,\skupotri^1,\skupotri^2$ are similar, and dual, for finite colimits.
\[
  \skuin(\omega) = \qeqi
\]
\[
  \sktrir(\skupotri^{\lambda}(\omega)) =
    \qeqinj{\sktril(\skupotri^1(\omega))}{\sktril(\skupotri^2(\omega))}^{\lambda}
  \quad
  (\lambda = 1,2)
\]

$\skul$, for \emph{list universals},
is novel, but works on similar principles.
For a list universal $\omega\in\skul$,
$\skule(\omega)$ and $\skulcons(\omega)$ supply the primary structure morphisms
$\qeqle$ and $\qeqlcons$
for $\qeql(A(\omega))$, where $A(\omega)=\skecod(\sktril(\skupbtri^1(\skulpb(\omega))))$.
The domains of the structure morphisms
($1$ and $A(\omega)\times\qeql(A(\omega))$) are limits,
and $\skult,\skulpb$ supply universals to stipulate them.
Note that, since we need a terminal anyway,
we might as well reuse it as the terminal needed for a product as special case of pullback.
The implied relations, which are in addition to those already implied
for $\skult(\omega)$ and $\skulpb(\omega)$, are --
\[
  \qeqle(A(\omega)) = \skule(\omega) \quad \qeqlcons(A(\omega)) = \skulcons(\omega)
\]

\subsection{Models}\label{sec:Models}
\begin{definition}\label{def:model}
  Let $\thT$ be a sketch and $\catA$ an AU.

  A \emph{strict model} of $\thT$ in $\catA$
  is an interpretation of nodes and edges in $\thT$
  as objects (carriers) and morphisms (operations) in $\catA$,
  in a way that respects all the implied relations of the sketch strictly,
  i.e. \emph{up to equality}.

  A \emph{model} of $\thT$ in $\catA$
  is an interpretation of nodes and edges in $\thT$ as objects and morphisms in $\catA$,
  in a way that respects up to equality all the domains, codomains, identities and
  commutativities of the sketch, and up to isomorphism all the universals.
  In other words, the subjects of each universal have to have the appropriate
  universal property, but do not have to be the canonical construction.

  A homomorphism between models of $\thT$ in an AU $\catA$
  comprises a carrier morphism for each node,
  together commuting with the operations in the appropriate way.
  This can be conveniently expressed
  as a model of $\thT$ in the comma category $\catA\downarrow\catA$, also an AU.
  (See~\cite{ArithInd} for results concerning these comma categories and their AU structure,
  and also for the related pseudopullback $\catA\downarrow_{\cong}\catA$.
  )

  We write $\Mod{\thT}{\catA}$ for the category of models of $\thT$ in $\catA$,
  and $\Mods{\thT}{\catA}$ for the full subcategory of strict models.

  If $h\colon\catA\to\catB$ is an AU-functor, then we obtain a functor
  \[
    \Mod{\thT}{h} \colon \Mod{\thT}{\catA} \to \Mod{\thT}{\catB}
    \text{.}
  \]
  If $h$ is a strict AU-functor, then $\Mod{\thT}{h}$ preserves strictness of models.
\end{definition}

As we remarked earlier, any sketch $\thT$ can be treated
as generators and relations for presenting an arithmetic universe,
using the fact that the theory of AUs is cartesian (see~\cite{PHLCC}).
We shall write this as $\AUpres{\thT}$.
It is the AU version of the notion of classifying category,
and we shall call it the \emph{classifying AU} for $\thT$.
It is the analogue of the classifying topos when geometric logic is replaced by an arithmetic form.

The injection of generators provides a strict \emph{generic model} $M_G$ of $\thT$
in $\AUpres{\thT}$,
and then the universal property is that any strict model $M$ of $\thT$ in an AU $\catA$
extends uniquely to a strict AU-functor $h\colon\AUpres{\thT} \to \catA$
for which $\Mod{\thT}{h}$ transforms $M_G$ to $M$ -- up to equality.
(This is analogous to the universal property for classifying toposes,
with strict AU-functors corresponding to the inverse image parts of geometric morphisms,
but note that the AU property is stricter.)

Thus strict models of $\thT$ are in bijection with strict AU-functors out of $\AUpres{\thT}$.
We have already seen that a non-strict AU functor out of $\AUpres{\thT}$
will also give rise to a non-strict model of $\thT$,
the non-strict image of the generic model.
However, the universal property does not allow us to recover the non-strict AU-functor from the model.
Hence the universal algebra is less precise for non-strict models and AU-functors.
In Section~\ref{sec:Extensions} we restrict the notion of sketch in a way that gives better control
over the non-strict models.

\begin{definition}\label{def:reduct}
  Let $f\colon\thT_1\skhom\thT_0$ be a homomorphism of sketches,%
  \footnote{
    Why this order of 1 and 0?
    Because in Section~\ref{sec:ConMaps} we shall think of $f$ as a \emph{map}
    from the space of models of $\thT_0$ to that of $\thT_1$,
    acting by model reduction.
  }
  and $M$ a model of $\thT_0$ in $\catA$.
  Then the \emph{$f$-reduct} of $M$, written $M|f$,
  is the model of $\thT_1$ whose carriers and operations are got by taking those
  for $M$ corresponding by $f$.

  It is a model because the sketch homomorphism transforms all the implied relations of
  $\thT_1$ into implied relations of $\thT_0$.
\end{definition}

Model reduction is functorial with respect to model homomorphisms,
and so the assignment $\thT\mapsto \Mod{\thT}{\catA}$
is the object part of a contravariant category-valued functor
$\Mod{(-)}{\catA}$ on $\mathbf{Sk}_{\skhom}$,
with sketch homomorphisms assigned to model reduction.

Model reduction preserves strictness.

By taking the $f$-reduct of the generic model in $\AUpres{\thT_0}$,
we get a strict model of $\thT_1$ in $\AUpres{\thT_0}$
and hence a strict AU-functor
$\AUpres{f} \colon \AUpres{\thT_1} \to \AUpres{\thT_0}$.

\subsection{Examples of sketches}\label{sec:SkExx}
Here are some examples of sketches.
Again, the notation is adapted to thinking of the sketch as prescribing a class of models
in each AU.

\begin{enumerate}
\item
  The empty sketch $\thone$ has a unique model in any AU.
\item
  The sketch $\thob$ has a single node and its identity edge and nothing else.
  Its models in $\catA$ are the objects of $\catA$.
\item
  Let $\thT$ and $\thU$ be two sketches.
  Their \emph{disjoint union} is called the \emph{product} sketch $\thT\times\thU$.
  Its models are pairs of models of $\thT$ and $\thU$.
  We also use notation such as $\thT^2$ for $\thT\times\thT$.
\item
  Let $\thT$ be a sketch.
  The \emph{hom sketch} $\thT^\to$ is made as follows.
  First, take two disjoint copies of $\thT$ as in $\thT^2$,
  distinguished by subscripts 0 and 1.
  These give two sketch homomorphisms $i_0,i_1 \colon \thT \to \thT^{\to}$.
  Next, for each node $X$ of $\thT$, adjoin an edge $\theta_X\colon X_0 \to X_1$;
  and, for each edge $u\colon X\to Y$ of $\thT$,
  adjoin an edge $\theta_u \colon X_0 \to Y_1$
  together with two commutativities to make a commutative diagram
  \[
    \xymatrix{
      {X_0}
        \ar@{->}[r]^{\theta_X}_{\bullet}
        \ar@{->}[dr]^{\theta_u}
        \ar@{->}[d]_{u_0}
      & {X_1}
        \ar@{->}[d]^{u_1}
      \\
      {Y_0}
        \ar@{->}[r]_{\theta_Y}^{\bullet}
      & {Y_1}
    }
  \]
  Then a model of $\thT^\to$ comprises a pair $M_0,M_1$ of models of $\thT$,
  together with a homomorphism $\theta\colon M_0 \to M_1$.

  The assignment $\thT \mapsto \thT^{\to}$ extends functorially to sketch homomorphisms,
  and then $i_0$ and $i_1$ become natural transformations.
\item
  We shall also write $\thT^{\to\to}$ for the theory of composable pairs of homomorphisms of $\thT$-models,
  and analogously for greater numbers of arrows.
  In fact, for any finite%
  \footnote{
    Actually, finiteness is not important here, as we have not set any finiteness
    conditions on the sketch $\thT$.
    But it will be important for contexts.
  }
  category $\catC$ we can write $\thT^\catC$
  for the theory of $\catC$-diagrams of models of $\thT$.
\end{enumerate}

The existence of $\thT^{\to}$ enables us to define \emph{2-cells} in $\mathbf{Sk}_{\skhom}$.
If $f_0,f_1\colon\thT_1\skhom\thT_0$,
then a 2-cell from $f_0$ to $f_1$ is a sketch homomorphism
$\alpha\colon\thT_1^\to \skhom\thT_0$
such that $i_{\lambda}\alpha = f_{\lambda}$ ($\lambda = 0,1$).
We also say that $\alpha$ is \emph{between} $\thT_0$ and $\thT_1$.

2-cells cannot yet be composed, either vertically or horizontally,
because edges cannot be composed in sketches.
However, we do have \emph{whiskering} on both sides,
using either $\alpha f$ or $f^{\to}\alpha$,
and it has all relevant associativities.

We can also take reducts along 2-cells.
If $M$ is a model of $\thT_0$ in $\catA$,
then the homomorphism $M|\gamma \colon M|f_0 \to M|f_1$
uses the carrier functions of $\thT_1^\to$ as interpreted in $\thT_0$.

\section{Extensions, contexts}\label{sec:Extensions}
In this section we define a class of sketches, the \emph{contexts},
for which every non-strict model can be made strict in a unique way.

What makes this non-trivial is that in general,
strictness has the ability to assert \emph{equalities} between sorts
by making a single node $X$ the subject of two different universals,
for example making it both $A \times B$ and $\List C$.
In a non-strict model this just requires $A \times B \cong \List C$,
whereas strictness would require equality;
and in an AU it can easily happen that the first holds but not the second.
Such equalities are not really the concern of category theory,
so better would be to have universals specifying two nodes $X_1$ and $X_2$
as $A \times B$ and $\List C$ respectively,
and then to specify an isomorphism $X_1\cong X_2$.
Strict models of that are unproblematic.

To enforce the latter kind we shall use each universal with a simple definitional effect,
defining its subjects fresh from some other ingredients (nodes and edges)
defined previously.
This leads to our notion of \emph{extension} of sketches.
To prepare for this, we introduce a notion of \emph{protoextension},
in which the syntactic notion of freshness is represented using categorical coproducts.

We say that a set is \emph{strongly finite} if it is
isomorphic to a finite cardinal $\{1,\ldots,n\}$ for some $n\in\mathbb{N}$.
Equivalently, it is Kuratowski finite, has decidable equality,
and can be equipped with a decidable total order.

\begin{definition}\label{protoext}
  A sketch homomorphism $i'\colon\thU\skhom\thU'$ is a \emph{protoextension} if
  for each sketch sort $\Xi$, we have that $\thU'_{\Xi}$ can be expressed as a coproduct
  $\thU_{\Xi} + \delta\Xi$,
  with $i'_{\Xi}$ a coproduct injection and $\delta\Xi$ strongly finite.
\end{definition}

\begin{proposition}\label{prop:protoext}
  Let $i'\colon\thU\skhom\thU'$ be a sketch homomorphism.
  Then the following are equivalent.
  \begin{enumerate}
  \item
    $i'$ is a protoextension.
  \item
    $i'$ is a pushout of some \emph{strongly finite sketch inclusion},
    by which we mean a sketch monomorphism $i\colon \thT\skhom\thT'$ in which
    $\thT$ and $\thT'$ are strongly finite (i.e. their carriers are).
  \end{enumerate}
\end{proposition}
\begin{proof}
  (2) $\Rightarrow$ (1):
  Let $i\colon \thT \skhom \thT'$ be a strongly finite sketch inclusion.
  For each sketch sort $\Xi$, we can write $\thT'_{\Xi}$ as a coproduct
  $\thT'_{\Xi} = \thT_{\Xi} + \delta\Xi$.
  (Informally in such a situation, we shall often write $\thT'$ as $\thT+\delta\thT$,
  although this is not a coproduct of sketches.
  $\delta\thT$ is not a sketch in its own right, as some of its structure may lie in $\thT$.)

  Now let $f\colon \thT \skhom \thU$ be an arbitrary sketch homomorphism.
  Then the pushout $i'\colon\thU\skhom\thU'$ of $i$ along $f$ can be constructed as follows.

  For each sketch sort $\Xi$, we let $\thU'_{\Xi} = \thU_{\Xi}+\delta\Xi$.
  For elements of $\thU_{\Xi}$, their structure is determined as in $\thU$.
  Now suppose $\omega\in\delta\Xi$.
  In $\thT+\delta\thT$, each structural element of $\omega$
  (i.e. the result of applying a sketch operator)
  is in either $\thT$ or $\delta\thT$.
  If the latter, then we keep it there in $\thU'$.
  If the former, then we apply $f$ to get it in $\thU$.
  We obtain a commutative diagram of sketches that is readily verified to be a pushout:
  \[
    \xymatrix{
      {\thU+\delta\thT}
      & {\thT+\delta\thT}
        \ar@{->}[l]_{f+\delta\thT}
      \\
      {\thU}
        \ar@{->}[u]^{i'}
      & {\thT}
        \ar@{->}[l]^{f}
        \ar@{->}[u]_{i}
    }
  \]
  From the construction, $i'$ is clearly a protoextension.

  (1) $\Rightarrow$ (2):
  Use the elements of the $\delta\Xi$s as generators for a sketch $\thT'$,
  with relations to say that the sketch operations in $\thU'$ are preserved
  \emph{insofar as} they stay in the $\delta\Xi$s.
  Then $\thT'$ is strongly finite,
  and the inclusion of the $\delta\Xi$s in $\thU'$ induces a sketch homomorphism
  $f'\colon\thT'\skhom\thU'$.

  Let $\thT$ be the pullback of $i'$ and $f'$, with projections $i$ and $f$.
  $i$ is monic, because $i'$ is.
  Also, in a coproduct the images of the injections are decidable subobjects,
  and it follows that the carriers of $\thT$ are decidable subobjects of those of $\thT'$,
  and so $\thT$ too is strongly finite.

  Applying the construction of (2) $\Rightarrow$ (1), we recover $i'$.
\end{proof}

It was already clear from the definition that protoextensions are closed
under composition.
From Proposition~\ref{prop:protoext} it is also clear that protoextensions $i'$ are closed under
pushout along any sketch homomorphism $g$.
The pushout is called the \emph{reindexing} of $i'$ along $g$, and written $g(i')$.

\subsection{Extensions: the definition}\label{sec:ExtDef}
In the following definition, central to the whole paper,
we restrict our proto-extensions by restricting the strongly finite sketch inclusions $i$
of Proposition~\ref{prop:protoext}.
First we define a finite family of inclusions $i\colon \thT \skhom \thT+\delta\thT$
that are generic for \emph{simple extensions},
and then a general \emph{extension} (written $\skext$) is a  composite of simple extensions.

For each kind of simple extension, using an inclusion $i$,
the sketch homomorphism $f\colon\thT\skhom\thU$ that we reindex along
can be understood as a \emph{data configuration} in $\thU$,
some tuple of elements satisfying some equations.
Thus each kind of simple extension can be understood as a sketch transformation
that takes \emph{data} (given by $f$) and delivers a \emph{delta}, according to Proposition~\ref{prop:protoext}.

Since any sketch homomorphism will transform extension data to extension data,
we see that reindexing (as sketch pushout) is got by applying the same extension to the
transformed data.
For an extension $c\colon \thT_1 \skext \thT'_1$, we shall typically write a reindexing square as
\begin{equation}\label{eq:reindex}
  \xymatrix{
    {\thT'_0}
    & {\thT'_1}
      \ar@{->}[l]_{\varepsilon}
    \\
    {\thT_0}
      \ar@{->}[u]^{f(c)}
    & {\thT_1}
      \ar@{->}[u]_{c}
      \ar@{->}[l]^{f}
  }
  \text{.}
\end{equation}

\begin{definition}\label{def:extn}
  A \emph{simple extension} is a proto-extension got as a pushout of one of the following
  strongly finite sketch inclusions $i\colon\thT\skhom\thT+\delta\thT$.
  Where we don't specify $\delta \Xi$, it is empty.
  \begin{enumerate}
  \item
    (Adding a new \emph{primitive node})
    No data (i.e. $\thT$ is $\thone$).
    Deltas:
    \begin{align*}
      \delta \skn &= \{\ast\} \\
      \delta \ske &= \{\sknid(\ast)\}
    \end{align*}
  \item
    (A simple \emph{functional} extension, by a new \emph{primitive edge})
    Data: $(X,Y)\in\skn \times \skn$.
    Delta:
    \[
      \delta \ske = \{\xymatrix@1{{X} \ar@{.>}[r] & {Y}}\}
      \text{.}
    \]
    In other words $\delta \ske = 1 = \{\ast\},\skedom(\ast) = X$, $\skecod(\ast) = Y$.
    We shall use similar informal notation in the other cases.
    Note that the ``delta'' edges are shown dotted.
  \item
    (Adding a commutativity)
    Data:
    $\xymatrix@1{
      {}
        \ar@{->}[r]_{u}
        \ar@{->}@/^0.5pc/[rr]^{w}
      & {}
        \ar@{->}[r]_{v}
      & {}
    }$.
    Delta:
    \[
      \delta\sktri = \{
        \xymatrix@1{
          {}
            \ar@{->}[r]_{u}
            \ar@{->}@/^1pc/[rr]^{w}_{\bullet}
          & {}
            \ar@{->}[r]_{v}
          & {}
        }
      \}\quad (uv\skdiag w)
      \text{.}
    \]
    In other words $\delta\sktri = \{\ast\}$ with
    $\sktril(\ast) = u$, $\sktrir(\ast) = v$, $\sktric(\ast) = w$.
  \item
    (Adding a terminal)
    No data. Deltas:
    \begin{align*}
      \delta \skut &= \{\ast\} \\
      \delta \skn &= \{\skutn(\ast)\} \\
      \delta \ske &= \{\sknid(\skutn(\ast))\}
    \end{align*}

    Adding an initial object is similar.
  \item
    (Adding a pullback)
    Data:
    $\xymatrix{
      {}
        \ar@{->}[r]^{u_1}
      & {}
      & {}
        \ar@{->}[l]_{u_2}
    }$.
    Deltas:
    \begin{align*}
      \delta\skupb &= \left\{
      \raisebox{1.5\height}{\xymatrix{
        {\mathsf{P}}
          \ar@{.>}[r]^{\mathsf{p}^2}_{\bullet}
          \ar@{.>}[dr]^{\mathsf{p}}
          \ar@{.>}[d]_{\mathsf{p}^1}
        & {}
          \ar@{->}[d]^{u_2}
        \\
        {}
          \ar@{->}[r]_{u_1}^{\bullet}
        & {}
      }}
      \right\} \\
      \delta\sktri &= \{ \mathsf{p}^1 u_1 \skdiag \mathsf{p},
        \mathsf{p}^2 u_2 \skdiag \mathsf{p} \} \\
      \delta\ske &=
        \{ \mathsf{p}^1, \mathsf{p}, \mathsf{p}^2, \sknid(\mathsf{P}) \} \\
      \delta\skn &= \{ \mathsf{P} \}
    \end{align*}
    Adding a pushout is similar.
  \item
    (Adding a list object)
    Data: $A\in\skn$.
    Deltas:
    \begin{align*}
      \delta \skul &=
        \{ \ast =
          (\xymatrix@1{{T} \ar@{.>}[r]^{\qeqle} & {L} & {P} \ar@{.>}[l]_{\qeqlcons}})
        \} \\
      \delta \skut &= \{ \skult(\ast) = T \} \\
      \delta \skupb &=
        \left\{ \skulpb(\ast) =
          \raisebox{1.5\height}{\xymatrix{
            {P}
              \ar@{.>}[r]^{\mathsf{p}^2}_{\bullet}
              \ar@{.>}[dr]^{\mathsf{p}}
              \ar@{.>}[d]_{\mathsf{p}^1}
            & {L}
              \ar@{.>}[d]^{!_L}
            \\
            {A}
              \ar@{.>}[r]_{!_A}^{\bullet}
            & {T}
          }}
        \right\} \\
      \delta\sktri &= \{ \mathsf{p}^1 !_A \skdiag \mathsf{p},
                         \mathsf{p}^2 !_L \skdiag \mathsf{p} \} \\
      \delta\ske  &= \{ \qeqle, \qeqlcons,
        \mathsf{p}^1, \mathsf{p}, \mathsf{p}^2, !_A, !_L,
        \sknid(T), \sknid(L), \sknid(P)
      \} \\
      \delta\skn &= \{ T,L,P \}
    \end{align*}
  \end{enumerate}
  An \emph{extension} of sketches is a proto-extension that can be expressed
  as a finite composite of simple extensions.
  We write $\thT_1\skext\thT_2$.

  An \emph{AU-context} is an extension of the empty sketch $\thone$.
\end{definition}

\begin{proposition}
  Let $c\colon \thT_1 \skext \thT_2$ be an extension of sketches.
  Then for each fresh node or edge $\alpha$ in $\thT_2$ there is an AU expression
  $w_{\alpha}$, well defined from the structure of $\thT_2$,
  by which, in any strict model,
  the interpretation of $\alpha$ can be found from those of the primitives and $\thT_1$.
\end{proposition}
\begin{proof}
  By inspecting the cases, we see that for a simple extension each fresh node or edge
  can be described \emph{uniquely} in one of the following ways.

  For nodes:
  the node is primitive or takes one of the forms
  \[
    \begin{array}{lll}
      \skutn(\omega) & (\omega\in\skut) & \qeqt   \\
      \skuin(\omega) & (\omega\in\skui) & \qeqi \\
      \skedom\sktric\skupbtri_1(\omega) & (\omega\in\skupb)
        & \qeqpb{u_1}{u_2}   \\
      \skecod\sktric\skupotri_1(\omega) & (\omega\in\skupo)
        & \qeqpo{u_1}{u_2}  \\
      \skecod\skule(\omega) & (\omega\in\skul) & \qeql(A)
    \end{array}
  \]

  For edges:
  the edge is primitive or takes one of the forms
  \[
    \begin{array}{lll}
      \sknid(X) & (X\in\skn) & \qeqid(X) \\
      \sktric\skupbtri_1(\omega)\text{ or }\sktril\skupbtri_i(\omega)
        & (\omega\in\skupb)
        & \qeqproj{u_1}{u_2} \text{ or }  \qeqproj{u_1}{u_2}^{i} \\
      \sktric\skupotri_1(\omega)\text{ or }\sktrir\skupotri_i(\omega)
        & (\omega\in\skupo)
        & \qeqinj{u_1}{u_2} \text{ or }  \qeqinj{u_1}{u_2}^{i}  \\
      \skule(\omega)\text{, }\skulcons(\omega)
        &  (\omega\in\skul)
        &  \qeqle(A) \text{, } \qeqlcons(A)  \\
      \sktrir\skupbtri_1\skulpb(\omega)\text{ or }\sktrir\skupbtri_2\skulpb(\omega)
        &  (\omega\in\skul)
        &  \qeqtfill(A) \text{ or }\qeqtfill(\qeql(A))
    \end{array}
  \]
  These facts are preserved by subsequent simple extensions,
  since those forms are only introduced for fresh nodes or edges.
  It follows that the facts remain true for the composite extension.

  We can now apply an induction on the number of composed simple extensions,
  and use the equations for strict models that are imposed by the sketch structure.
  We look explicitly at universals for pullbacks and list.
  Other situations are similar or easier.

  First, consider a simple extension in the form of a pullback universal $\omega$,
  defined on the configuration
  $\xymatrix{
    {}
      \ar@{->}[r]^{u_1}
    & {}
    & {}
      \ar@{->}[l]_{u_2}
  }$.
  The relations for such a universal tell us that the fresh edge $\sktril\skupbtri_1(\omega)$
  has to be interpreted as $\qeqproj{u_1}{u_2}^{1}$,
  and we use induction to find the expressions for $u_1$ and $u_2$.
  (The base case is if they are primitive or in $\thT_1$.)
  The other fresh edges and the fresh node are dealt with in a similar way.
  Note that if $\omega=\skulpb(\omega')$ for some $\omega' \in \skul$,
  then the subjects of $\omega$ are treated in the same way,
  but $u_i = \sktrir(\skupbtri_i(\omega))$ gets its expression from $\omega'$.

  Now consider a simple extension in the form of a list universal $\omega$, on object $A$.
  All the fresh nodes and edges have expressions in terms of $A$.
  For $\skule(\omega)$ and $\skulcons(\omega)$ and their codomain this is clear.
  Next, from the terminal universal $\skult(\omega)$ we have $\skutn(\skult(\omega))=\qeqt$.
  Because this appears as a vertex in the pullback square $\skulpb(\omega)$,
  it follows from the AU axioms that
  $\sktrir(\skupbtri_1(\skulpb(\omega)))=\qeqtfill_A$ and
  $\sktrir(\skupbtri_2(\skulpb(\omega)))=\qeqtfill_{\qeql(A)}$.
  Since these are $u_1$ and $u_2$ in the treatment of the pullback universals,
  it only remains to deal with the easy case of the identity morphisms.
\end{proof}

Note that a primitive edge can acquire equality with an AU-expression by
subsequently added commutativities.
We shall use this later for introducing AU operators that have not been mentioned
so far in extensions.

\subsection{Strictness results}\label{sec:StrictResults}
The reason for introducing extensions was for an important property that
non-strict interpretations can be reinterpreted strictly in a unique way.
The following definition and lemma will make this precise,
albeit in a generality whose usefulness will only be seen in sequel papers.

\begin{definition}
Let $\thT\skext\thT'$ be a sketch extension.
A model of $\thT'$ is \emph{strict} for the extension if, for each universal,
each subject node or edge is equal to the result of its expression.
\end{definition}
Note that a model of $\thT'$ is strict in its own right iff it is strict for the extension
and its $\thT$-reduct is strict.

\begin{lemma}\label{lem:extReindex}
  Suppose, as in the diagram below, an extension $\thT_1\skext\thT'_1$
  is reindexed along a sketch homomorphism $\thT_1 \skhom \thT_0$.
  Suppose also that in some AU $\catA$ we have models $M_0$ and $M'_1$ of $\thT_0$ and $\thT'_1$,
  with an isomorphism $\phi\colon M_0|\thT_1 \cong M'_1 | \thT_1$.

  \[
    \xymatrix{
      {M'_0}
      & & {\thT'_0}
      & {\thT'_1}
        \ar@{->}[l]_{\varepsilon}
      & {M'_0|\thT'_1}
        \ar@{.>}[r]^{\phi'}_{\cong}
      & {M'_1}
      \\
      {M'_0|\thT_0}
        \ar@{=}[r]
      & {M_0}
      & {\thT_0}
        \ar@{->}[u]^{f(\skext)}
      & {\thT_1}
        \ar@{->}[l]^{f}
        \ar@{->}[u]_{\skext}
      & {M_0|\thT_1}
        \ar@{->}[r]^{\phi}_{\cong}
      & {M'_1|\thT_1}
    }
  \]

  Then there is a unique model $M'_0$ of $\thT'_0$
  and isomorphism $\phi'\colon M'_0 | \thT'_1 \cong M'_1$ such that
  \begin{enumerate}
  \item
    $M'_0|\thT_0 = M_0$,
  \item
    $M'_0$ is strict for the extension $\thT_0\skext\thT'_0$,
  \item
    $\phi' | \thT_1 = \phi$, and
\item
    $\phi'$ is equality on all the primitive nodes for the extension $\thT_1\skext\thT'_1$.
  \end{enumerate}
\end{lemma}
\begin{proof}
  It suffices to cover the cases for a simple extension $\thT_1\skext\thT'_1$.

  If the extension adjoins a primitive node $X$,
  then we can and must take its carrier in $M'_0$ to be equal to its carrier in $M'_1$,
  and the carrier function in $\phi'$ to be the identity.

  Suppose the extension adjoins a primitive edge $u\colon X \to Y$.
  Then $\phi'$ must equal $\phi$,
  and to preserve the homomorphism property we can and must define the operation for
  $u$ in $M'_0$ to be $\phi(X)M'_1(u)\phi^{-1}(Y)$, using the operation in $M'_1$.

  If the extension adjoins a new commutativity, then the morphism equation already holds
  in $M'_1|\thT_1$ and hence in $M_0$, so we can and must take $M'_0$ and $\phi'$ to be given
  by the same data as $M'_1$ and $\phi$.

  It remains only to examine the case where the extension adds a universal.
  We consider the case of a list universal,
  as the others are similar (and easier).
  $M'_0$ has to interpret the new nodes and edges in the canonical way.
  In particular, $T$, $L$ and $P$ are $\qeqt$, $\qeql(A)$ and $A\times\qeql(A)$.
  Then the universal properties (of terminal object, list object and binary product)
  give canonical isomorphisms between those canonical interpretations in $M'_0$
  and the corresponding interpretations (possibly non-canonical) in $M'_1$.
  The corresponding carrier morphisms of $\phi'$ can be defined to be
  those canonical isomorphisms,
  and indeed by the homomorphism properties and uniqueness of fillins they must be so defined.
\end{proof}

By considering the case where $\thT_0 = \thT_1 = \thone$, we obtain --

\begin{corollary}\label{cor:strictModelContext}
  Let $\thT$ be a context, $\catA$ an AU, and $M_1$ a model of $\thT$.
  Then there is a unique strict model $M_0$ of $\thT$ and isomorphism $\phi\colon M_0 \cong M_1$.
\end{corollary}

It follows that if $\thT$ is a context,
and $h\colon\catA\to\catB$ is a \emph{non-}strict AU-functor,
then we get a functor $\Mods{\thT}{h}\colon\Mods{\thT}{\catA}\to\Mods{\thT}{\catB}$.
Given a strict model $M$ in $\catA$, composing with $h$ gives a non-strict model in $\catB$,
and we can then take the unique strict model isomorphic to it.
An important example is when $h$ is the inverse image part of a geometric morphism
between two toposes with natural number objects.

\subsection{Examples of contexts}\label{sec:ConExx}
Here are some examples of contexts. (cf. Section~\ref{sec:SkExx}.)

\begin{enumerate}
\item 
  The sketches $\thone$ and $\thob$ are both contexts.
\item 
  If $\thT_0$ and $\thT_1$ are both contexts, then so is $\thT_0\times\thT_1$.
  To be specific, we shall adjoin the ingredients of $\thT_0$ first,
  so that $\thT_0\skext\thT_0\times\thT_1$ is an extension
  and for $\thT_1$ we just have a homomorphism $\thT_1\skhom\thT_0\times\thT_1$.
\item 
  If $\thT$ is a context, then so is $\thT^\to$.
  We take it that $i_0\colon\thT\skext\thT^{\to}$ is the extension.

  Similarly, $\thT^{\to\to}$ is a context, with extensions
  $\xymatrix@1{
    {\thT} \ar@{->}[r]^{i_0}_{\skext}
    & {\thT^{\to}} \ar@{->}[r]^{i_{01}}_{\skext}
    & {\thT^{\to\to}}
  }$,
  where $i_{01}$ is the reindexing $i_1(i_0)$.

  More generally, for any strongly finite
  category $\catC$ we have that $\thT^\catC$ can be made a context.
  The order of simple extensions for it will depend on a total order given to each finite set involved.
\item 
  If $\thT$ is a context, then it has an extension $\thT^{ns}$ whose strict models are the
  \emph{non-strict} models of $\thT$.
  For each non-primitive node $X$,
  we adjoin a primitive node $X'$ together with an isomorphism $X'\cong X$.

  Note that we do need $\thT$ to be a context here, not an arbitrary sketch.
  A model of $\thT^{ns}$ is actually an isomorphic pair of two models,
  one strict and the other not.
  We need Corollary~\ref{cor:strictModelContext} to get this pair from any non-strict model.
\item 
  Without going into details, there is a context $\mathbb{R}$ for the theory of Dedekind sections.
  It is defined as outlined in~\cite{ArithInd}.
  First, the natural numbers $\mathbb{N}$ can be defined as $\qeql(1)$.
  Their (decidable) order and arithmetic can be defined using the universal property.
  Then the rationals $\mathbb{Q}$ can be defined by standard techniques,
  together with their decidable order and arithmetic.
  Next, two nodes $L$ and $R$ are adjoined,
  with edges to $\mathbb{Q}$ and conditions to make them monic.
  Finally the various axioms for Dedekind sections are imposed.
\item 
  For various kinds of presentation of locales,
  there are context extensions $\thT_0\skext\thT_1$
  where a model of $\thT_0$ is a presentation,
  and one of $\thT_1$ is a presentation equipped with a point of the corresponding locale.

  The same principle also applies in formal topology,
  with an inductively generated formal topology understood as a presentation.

  For example, suppose we take the formal topologies as defined in~\cite{CSSV:IndGenFT}.
  First we declare the base $B$, a poset.
  Next, the cover $\vartriangleleft_0$ can be adjoined as a node,
  with an edge to $B$.
  A node $C$ is adjoined for a disjoint union of all the covering sets,
  with an edge to $\vartriangleleft_0$.
  The conditions on these can also be expressed using AU structure in a context $\thT_0$.
  For $\thT_1$ we adjoin to $\thT_0$ a monic into $B$,
  together with conditions to make it a formal point.

  Note that we have not attempted here to extract the full cover $\vartriangleleft$.
\end{enumerate}

\section{Equivalence extensions}\label{sec:EquivExt}

An equivalence extension is an extension,
but one in which the simple extension steps are grouped together in a way
that guarantees that the fresh ingredients
(nodes, edges, properties, equations) introduced in the extension are all
already known to exist uniquely.
The most intricate parts are for the edges.
In an ordinary extension, an unconstrained fresh edge can subsequently be
specified uniquely up to equality by commutativities (equations).
In an equivalence extension when we introduce an edge we must also document the
justification for its existence
(as a composite or a fillin;
universal structure edges such as limit projections are introduced along with the universal objects).
In addition, we must also include steps for proving equations between edges --
this is to provide images for commutativities under a sketch morphism.
These steps essentially codify the rules for congruences in universal algebra.
(The reason this is not needed for nodes is that essentially algebraic theories
of categories do not normally have any axioms to imply equations between objects.)

The game now is to describe simple equivalence extensions sufficient to
generate all the operators of the the theory of AUs
and all the arrow equalities generated by the axioms.
(For object equalities see Section~\ref{sec:ObjEq}.)

\begin{definition}\label{def:equivExt}
  A \emph{simple equivalence extension} is a proto-extension of one of
  the following forms (or \emph{rules}).
  Note that each is in fact an extension.

  In each case, every node or edge introduced will, in any strict model,
  become equal to a certain AU expression in terms of the data.
  For nodes, which are all introduced by simple extensions of universal kind,
  this has already been covered in Definition~\ref{def:extn}.
  For edges the expressions are given in $\delta\ske$.
  Those expressions do indeed satisfy the commutativities listed in $\delta\sktri$.
  On the other hand, any edges satisfying them will be equal to the expressions by the AU equations
  for uniqueness of fillins.

  First, there are various rules associated with morphisms and their composition.
  They are summarized in this table.

  \begin{tabular}{|c|c|l|}
    \hline
    Data & Delta & \\
    \hline
    $\xymatrix@1{ {} \ar@{->}[r]_{u} & {} \ar@{->}[r]_{v} & {} } $
    & $ \xymatrix@1{ {} \ar@{->}[r]_{u} \ar@{.>}@/^1pc/[rr]_{\bullet}^{v\qeqcomp u}
      & {} \ar@{->}[r]_{v}
      & {} } $
    & composition
    \\ \hline
    $ \xymatrix@1{ {X} \ar@{->}[r]_{u} & {Y} } $
    & $ \xymatrix@1{ {X} \ar@{->}[r]_{\sknid(X)} \ar@{->}@/^1pc/[rr]^{u}_{\bullet}
      & {X} \ar@{->}[r]_{u}
      & {Y} } $
    & left unit law
    \\ \hline
    $ \xymatrix@1{ {X} \ar@{->}[r]_{u} & {Y} } $
    & $ \xymatrix@1{{X} \ar@{->}[r]_{u} \ar@{->}@/^1pc/[rr]^{u}_{\bullet}
      & {Y} \ar@{->}[r]_{\sknid(Y)}
      & {Y} } $
    & right unit law
    \\ \hline
    $ \xymatrix@1{ {} \ar@{->}[r] \ar@{->}@/_1pc/[rr]^{\bullet} \ar@{->}@/^1.5pc/[rrr]_{\bullet}
      & {} \ar@{->}[r] \ar@{->}@/^1pc/[rr]_{\bullet}
      & {} \ar@{->}[r]
      & {}  } $
    & $ \xymatrix@1{ {} \ar@{->}@/_0.5pc/[rr] \ar@{->}@/^1pc/[rrr]_{\bullet}
      & {}
      & {} \ar@{->}[r]
      & {}  } $
    & left associativity
    \\ \hline
    $ \xymatrix@1{ {} \ar@{->}[r] \ar@{->}@/_1pc/[rr]^{\bullet} \ar@{->}@/_1.5pc/[rrr]^{\bullet}
      & {} \ar@{->}[r] \ar@{->}@/^1pc/[rr]_{\bullet}
      & {} \ar@{->}[r]
      & {}  } $
    & $ \xymatrix@1{ {} \ar@{->}[r] \ar@{->}@/_1pc/[rrr]^{\bullet}
      & {} \ar@{->}@/^0.5pc/[rr]
      & {}
      & {}  } $
    & right associativity
    \\ \hline
  \end{tabular}

  Second, for each kind of universal (terminal, pullback, initial, pushout, list),
  we have three rules.
  The first will be the simple extension that introduces the corresponding node,
  the second will introduce fillins by adjoining a primitive edge with the appropriate equations,
  and the third will introduce equations for the uniqueness of fillins.

  We illustrate this for pullbacks and for list objects.
  The rules for terminals, initials and pushouts follow the same principles as for pullbacks.

  For pullbacks:
  \begin{itemize}
  \item
    A simple extension for a pullback universal is also an equivalence extension.
  \item
    Suppose we have a pullback universal $\omega\in\skupb$,
    and another cone given as $\Delta_1,\Delta_2\in\sktri$,
    with
    \[\begin{split}
      \sktrir(\Delta_i) & = \sktrir(\skupbtri_i(\omega)) = u_i \\
      \sktric(\Delta_1) & = \sktric(\Delta_2) = v
      \text{.}
    \end{split}\]
    \[
      \omega \text{ is }
      \xymatrix{
        {\mathsf{P}}
          \ar@{->}[d]_{\mathsf{p}^1}
          \ar@{->}[dr]^{\mathsf{p}}
          \ar@{->}[r]^{\mathsf{p}^2}_{\bullet}
        & {}
          \ar@{->}[d]^{u_2}
        \\
        {}
          \ar@{->}[r]_{u_1}^{\bullet}
        & {}
      }
      \quad
      \Delta_1,\Delta_2 \text{ are }
      \xymatrix{
        {}
          \ar@{->}[d]_{v_1}
          \ar@{->}[dr]^{v}
          \ar@{->}[r]^{v_2}_{\bullet}
        & {}
          \ar@{->}[d]^{u_2}
        \\
        {}
          \ar@{->}[r]_{u_1}^{\bullet}
        & {}
      }
    \]
    Then our equivalence extension has
    \[\begin{split}
      \delta \ske & = \{ w = \qeqpbfill{v_1}{v_2}{u_1}{u_2} \} \\
      \delta \sktri & = \{ w\mathsf{p}^1 \skdiag v_1, w\mathsf{p}^2 \skdiag v_2 \}
      \text{.}
    \end{split}\]
  \item
    Suppose we have a pullback universal $\omega\in\skupb$ as above,
    and edges $v_1, v_2, w,w'$ with commutativities
    $w\mathsf{p}^1 \skdiag v_1, w\mathsf{p}^2 \skdiag v_2,
       w'\mathsf{p}^1 \skdiag v_1, w'\mathsf{p}^2 \skdiag v_2$.
    Then our equivalence extension has
    \[
      \delta \sktri = \{ w \skdiag w' \}
      \text{.}
    \]
  \end{itemize}

  For list objects:
  \begin{itemize}
  \item
    A simple extension for a list universal is also an equivalence extension.
  \item
    Suppose we have a list universal $\omega\in\skul$ with
    $\xymatrix@1{
      {1}
        \ar@{->}@<0.5ex>[r]^{\qeqle}
      & {L}
      & {A\times L}
        \ar@{->}[l]_{\qeqlcons}
    }$.
    Suppose (see diagram~\eqref{eq:listunichar}) we also have nodes $B,Y$,
    pullback universals to specify nodes for
    $L\times B$, $(A\times L)\times B$, $A\times(L\times B)$ and $A\times Y$,
    edges $\xymatrix@1{{B} \ar@{->}[r]^{y} & {Y} & {A\times Y} \ar@{->}[l]_{g}}$,
    and edges for $\qeqprodfill{\qeqtfill_B\qeqle}{B}$, $\qeqlcons\times B$
    and the associativity isomorphism,
    together with auxiliary edges and commutativities needed to characterize them.

    Using the notation of the following diagrams,
    our equivalence extension has
    $\delta\ske = \{ r, r', r'', g', g'' \}$,
    where $r = \qeqlrec{A}{y}{g}$,
    and $\delta\sktri$ comprises the seven commutativities shown.
    The second diagram is what is needed to specify that $r' = A \times r$.
    \begin{equation}\label{eq:eqExtLfill}
      \xymatrix@C=1cm{
        & {L \times B}
          \ar@{.>}[dd]_{r}^(0.8){\bullet}
        & {(A \times L)\times B}
          \ar@{->}[l]_-{\qeqlcons \times B}
          \ar@{.>}[ddl]_{g''}^(0.7){\bullet}
          \ar@{->}[d]^{\cong}
        \\
        & & {A\times(L \times B)}
          \ar@{.>}[dl]_{g'}^(0.7){\bullet}
          \ar@{.>}[d]^{r'=A\times r}
        \\
        {B}
          \ar@{->}[uur]^(0.8){\qeqprodfill{\qeqtfill_{B}\qeqle}{B}}
          \ar@{->}[r]_{y}^(0.8){\bullet}
        & {Y}
        & {A \times Y}
          \ar@{->}[l]^{g}
      }
      \xymatrix@C=0.7cm{
        & {A\times(L \times B)}
          \ar@{->}[r]^-{\mathsf{p}^2}
          \ar@{.>}[dr]^{r''}
          \ar@{.>}[d]^{r'}
          \ar@{->}[dl]_{\mathsf{p}^1}
        & {L \times B}
          \ar@{.>}[d]^{r}_{\bullet}
        \\
        {A}
        & {A \times Y}
          \ar@{->}[l]^{\mathsf{p}^1}_{\bullet}
          \ar@{->}[r]_{\mathsf{p}^2}^{\bullet}
        & {Y}
      }
    \end{equation}
  \item
    Suppose, given the configuration for the above fillin,
    we have two solutions with fillins $r_1, r_2$.
    Then our equivalence extension has
    \[
      \delta\sktri = \{ r_1 \skdiag r_2 \}
      \text{.}
    \]
    (Equivalence of the other edges can then be deduced.)
  \end{itemize}

  Finally, we have rules for balance, stability and exactness.
  In each case, the given configuration contains a particular edge $u\colon X \to Y$
  for which the equivalence extension adjoins an inverse. Hence
  \[
    \delta\ske = \{ u^{-1} \} \text{,} \quad
    \delta\sktri = \{ u u^{-1} \skdiag \sknid(X), u^{-1}u \skdiag \sknid(Y) \} \text{.}
  \]
  \begin{itemize}
  \item
    Rule for balance.
    Suppose we are given pullback and pushout universals
    $\omega\in\skupb, \omega'\in\skupo$,
    expressing the kernel pair and cokernel pair for the same edge $u\colon X \to Y$.
    \[
      \omega \text{ has }
      \xymatrix{
        {\qeqpb{u}{u}}
          \ar@{->}[r]^{\qeqproj{u}{u}^2}
          \ar@{->}[d]_{\qeqproj{u}{u}^1}
        & {X}
          \ar@{->}[d]^{u}
        \\
        {X}
          \ar@{->}[r]_{u}
        & {Y}
      }
      \text{,}\quad
      \omega' \text{ has }
      \xymatrix{
        {X}
          \ar@{->}[r]^{u}
          \ar@{->}[d]_{u}
        & {Y}
          \ar@{->}[d]^{\qeqinj{u}{u}^2}
        \\
        {Y}
          \ar@{->}[r]_{\qeqinj{u}{u}^1}
        & {\qeqpo{u}{u}}
      }
      \text{.}
    \]
    Suppose we also have commutativities $\qeqproj{u}{u}^1\skdiag \qeqproj{u}{u}^2$ ($u$ is monic)
    and $\qeqinj{u}{u}^1 \skdiag \qeqinj{u}{u}^2$ ($u$ is epi).
    Then our equivalence extension has $\delta\ske = \{ u^{-1} = \qequc(u) \}$.
  \item
    Rule for stability of initial objects.
    Suppose we are given a universal $\omega\in\skui$ for an initial object $0$,
    and an edge $u\colon X \to 0$.
    Then $\delta\ske = \{ u^{-1} = \qeqifill_X \}$.
  \item
    Rule for stability of pushouts.
    Suppose we have data as outlined in diagram~\eqref{eq:poStab}.
    This will include two pushout universals (bottom square and inner square on top),
    three pullback universals for vertical squares
    (front and right faces, and also one stretching diagonally over $v$),
    the extra edge $e$, and other diagonal edges where necessary.
    Then the equivalence extension inverts $e$,
    $\delta\ske = \{ e^{-1} = \qeqpostab{u_1}{u_2}{w} \}$.
  \item
    Rule for exactness.
    Suppose we have data as outlined in diagram~\eqref{eq:exactness}.
    This will include pullback universals to specify that $X_0 \times X_0$,
    $X_2$ and $K$ are the appropriate limits,
    pushout universals to specify that $\gamma$ is a coequalizer,
    and commutativities to specify that $\pi$ and $e$ are fillins.
    Then the equivalence extension inverts $e$,
    $\delta\ske = \{ e^{-1} = \qeqex(\pi_1,\pi_2,r,s,t) \}$.
  \end{itemize}

  An \emph{equivalence extension}, written $\thT\skeqext\thT'$,
  is a proto-extension that can be expressed as a composite of finitely
  many simple equivalence extensions.
\end{definition}

Note also that if $\thT\skeqext\thT'$ is an equivalence extension,
then so too is its reindexing along any sketch homomorphism.

If $e_i \colon \thT \skeqext \thT_i$ ($i=1,2$) are two equivalence extensions of a context $\thT$,
then $e_2$ is a \emph{refinement} of $e_1$, by $\varepsilon$,
if $\varepsilon\colon\thT_1\to\thT_2$ is a homomorphism such that $e_1\varepsilon = e_2$.

For any two equivalence extensions $e_i$ of $\thT$,
we can reindex $e_2$ along $e_1$ (or vice versa),
compose, and thereby get a common refinement of $e_1$ and $e_2$.

Equality between morphisms $u,u'\colon X \to Y$
is expressed using unary commutativities $u\skdiag u'$,
defined as $\sknid(X)u\skdiag u'$.
Since the rules used in equivalence extensions must be capable of supplying proofs of equality,
we verify that the standard rules for equality can be derived as composite rules of equivalence extensions.
Given these, it will be clear that all proofs of equality of morphisms in the essentially algebraic
theory of categories can be represented by commutativities in a suitable equivalence extension.

\begin{proposition}\label{prop:eqExtEqLogic}
  Let $\thT$ be a sketch.
  In the following results we are interested in properties holding in $\thT$,
  and properties derivable from them in the sense that they hold in some equivalence
  extension of $\thT$.
  \begin{enumerate}
  \item
    For any two nodes $X$ and $Y$,
    $\skdiag_{XY}$ is an equivalence relation on the edges between them.
    This is in the sense that for each of the three properties for an equivalence relation,
    if the hypothesis holds in some sketch then the conclusion holds in some equivalence extension.
  \item
    If $u,u'$ are two edges from $X$ to $Y$,
    then the commutativities $\sknid(X)u\skdiag u'$ and $u'\sknid(Y) \skdiag u$
    are mutually derivable.

    It follows that we have four mutually derivable characterizations of $u\skdiag u'$,
    namely $\sknid(X) u \skdiag u'$, $\sknid(X) u' \skdiag u$,
    $u \sknid(Y) \skdiag u'$ and $u' \sknid(Y) \skdiag u$.
  \item
    Suppose we have $u\skdiag_{XY}u'$ and $v\skdiag_{YZ} v'$,
    and also $w,w'\colon X \to Z$ with $uv\skdiag w$.
    Then the commutativities $u'v'\skdiag w'$ and $w\skdiag w'$
    are mutually derivable.

    From left to right is congruence.
    From right to left (with $u'=u$ and $v'=v$) shows that the set of
    composites $uv$ is the entire congruence class of $w$.
  \end{enumerate}
\end{proposition}
\begin{proof}
  (1)
  \emph{Reflexivity} is immediate from the left unit law.

  For \emph{symmetry,} suppose $u\skdiag_{XY} u'$.
  By the left unit law we derive
  \[ \xymatrix@1{
    {X}
      \ar@{->}[r]^{\sknid(X)}_{\bullet}
      \ar@{->}@/_1pc/[rr]_(0.8){\sknid(X)}
      \ar@{->}@/_1.5pc/[rrr]_(0.8){u}^(0.8){\bullet}
    & {X}
      \ar@{->}[r]^{\sknid(X)}
      \ar@{->}@/^1.5pc/[rr]^(0.8){u'}_{\bullet}
    & {X}
      \ar@{->}[r]^{u}
    & {Y}
  } \text{,}\]
  and then right associativity gives $\sknid(X)u'\skdiag u$.

  For \emph{transitivity,}
  suppose $u\skdiag u' \skdiag u''$.
  By the left unit law we get
  \[ \xymatrix@1{
    {X}
      \ar@{->}[r]_{\sknid(X)}^{\bullet}
      \ar@{->}@/_1.5pc/[rr]_(0.8){\sknid(X)}^{\bullet}
      \ar@{->}@/^1.5pc/[rrr]^(0.2){u''}
    & {X}
      \ar@{->}[r]_{\sknid(X)}
      \ar@{->}@/^1pc/[rr]^(0.2){u'}_{\bullet}
    & {X}
      \ar@{->}[r]_{u}
    & {Y}
  } \text{,}\]
  and then $\sknid(X)u\skdiag u''$ is derived by left associativity.

  (2)
  The two directions follow by applying associative laws to the two diagrams
  \[
    \xymatrix@1{
      {X}
        \ar@{->}[r]_{\sknid(X)}^{\bullet}
        \ar@{->}@/_1.5pc/[rr]_(0.8){u'}^{\bullet}
        \ar@{->}@/^1.5pc/[rrr]^(0.2){u}
      & {X}
        \ar@{->}[r]_{u}
        \ar@{->}@/^1pc/[rr]^(0.2){u}_{\bullet}
      & {X}
        \ar@{->}[r]_{\sknid(Y)}
      & {Y}
    }
    \text{ and }
    \xymatrix@1{
      {X}
        \ar@{->}[r]^{\sknid(X)}_{\bullet}
        \ar@{->}@/_1pc/[rr]_(0.8){u'}
        \ar@{->}@/_1.5pc/[rrr]_(0.8){u'}^(0.8){\bullet}
      & {X}
        \ar@{->}[r]^{u'}
        \ar@{->}@/^1.5pc/[rr]^(0.8){u}_{\bullet}
      & {X}
        \ar@{->}[r]^{\sknid(Y)}
      & {Y}
    }
  \]

  (3)
  First, consider the case when $v=v'$, and the diagram
  \[
    \xymatrix@1{
      {X}
        \ar@{->}[r]^{\sknid(X)}
        \ar@{->}@/_1pc/[rr]_{u'}^{\bullet}
      & {X}
        \ar@{->}[r]_{u}
        \ar@{->}@/^1pc/[rr]^{w}_{\bullet}
      & {Y}
        \ar@{->}[r]_{v}
      & {Z}
    }
    \text{.}
  \]
  The two associativities give the two directions we want.
  A similar proof, but dual (using (2)), deals with the case $u=u'$.
  Putting these together gives the general result.
\end{proof}

\begin{proposition}\label{prop:eeModelExtn}
  Let $\thT\skeqext\thT'$ be an equivalence extension,
  and let $M$ be a strict model of $\thT$ in an AU $\catA$.

  Then there is a unique strict model $M'$ of $\thT'$ in $\catA$
  whose restriction to $\thT$ is $M$.

  We call this the \emph{extension} of $M$ to $\thT'$.
\end{proposition}
\begin{proof}
  Each node or edge introduced in $\thT'$ has a canonical description
  as an AU-expression in terms of older nodes and edges,
  and so has a canonical interpretation already in $\AUpres{\thT}$.
  For a node, strictness implies already that we must use this interpretation.
  For an edge, the commutativities introduced at the same time are enough to force
  equality in $\catA$ between the interpretations of the edge and the canonical description.

  It remains to show that all the commutativities $fg\skdiag h$ in $\thT'$ are respected.
  Let us write $[f]$ for the interpretation of the canonical expression for $f$ in $M$,
  and similarly for $g$ and $h$.
  We require $[h] = [g]\qeqcomp[f]$ in $\catA$.

  We have to examine the rule that introduces the commutativity,
  and use induction on the number of simple equational extensions needed.

  For the rules that introduce nodes or edges,
  the commutativities introduced follow directly from quasiequational rules for AUs.
  It is also clear for unit laws and associativity.

  There remain the uniqueness rules for fillins.
  Suppose we have one expressing $f\skdiag f'$.
  Then in $\catA$ we have that $f$ and $f'$ are both equal to the fillin,
  and so equal to each other.
\end{proof}

\begin{proposition}\label{prop:eqExtAUiso}
  Let $e\colon\thT\skeqext\thT'$ be an equivalence extension.
  Then the corresponding AU-functor
  $\AUpres{e} \colon \AUpres{\thT}\to\AUpres{\thT'}$ is an isomorphism.
\end{proposition}
\begin{proof}
  In terms of strict AU-functors, Proposition~\ref{prop:eeModelExtn} says that
  for any strict $M\colon\AUpres{\thT}\to\catA$ there is a unique strict
  $M'\colon\AUpres{\thT'}\to\catA$ such that $\AUpres{e}M' = M$.
  Applying this with $\Id_{\AUpres{\thT}}$ for $M$ gives us an
  $F\colon \AUpres{\thT'}\to\AUpres{\thT}$ for $M'$,
  and then for more general $M$ we see that $M' = FM$.
  From this we deduce that $F$ is an inverse for $\AUpres{e}$.
\end{proof}

\section{Object equalities}\label{sec:ObjEq}
The notion of equality between two context homomorphisms (Definition~\ref{def:AUSk}) is very strong,
and in essence syntactic.
The homomorphisms must act equally on the nodes and edges as sketch ingredients.
In practice we usually want a more semantic notion that allows us to say when nodes and edges are equal
in the sense that they must be interpreted equally in strict models.
This will allow us to get faithfulness for a functor that takes $\thT$ to $\AUpres{\thT}$.

For edges, we already have a machinery for proving equality as morphisms by using commutativities.
For nodes we have deliberately avoided anything analogous,
beyond equality in the graph.
However, semantic equality can still arise when two nodes are declared by universals
for two identical constructions from equal data.
We define certain kinds of edges as being ``object equalities'' between their domains and codomains;
semantically they must be equal to identity morphisms.
We then extend the phrase to apply also to ``object equality'' between  edges or context homomorphisms.

We use the phrase \emph{object equality} for a situation
where a context already has the required structure,
and \emph{objectively equal}, or \emph{objective equality},
for a situation where an equivalence extension can provide it.

\begin{definition}\label{def:objeq}
  Let $\thT$ be a context,
  and suppose $\gamma\colon X \to Y$ is an edge in $\thT$.
  Then $\gamma$ is an \emph{object equality},
  written $\gamma\colon X \Rightarrow Y$, if
  either $X=Y$ as nodes and $\gamma\skdiag\sknid(X)$ in $\thT$,
  or $\gamma$ can be provided with structure in $\thT$ in one of the following ways.
  \begin{enumerate}
  \item
    If $X,Y$ are subjects of terminal universals: no extra structure needed.
  \item
    Suppose $X$ and $Y$ are subjects of pullback universals,
    for the back and front faces of the following diagram,
    and suppose also we have we have object equalities
    $\gamma_i\colon U_i \Rightarrow V_i$
    and edges and commutativities to name such composites as are required
    and to assert $u_i\gamma_3 \skdiag \gamma_i v_i$ and their consequence
    $p_1\gamma_1 v_1 \skdiag p_2 \gamma_2 v_2$,
    and $\gamma q_i \skdiag p_i \gamma_i$
    (characterizing $\gamma$ as a fillin
    $\qeqpbfill{p_1\gamma_1}{p_2\gamma_2}{v_1}{v_2}$).

    Then $\gamma \colon \colon X \Rightarrow Y$ is an object equality.
    
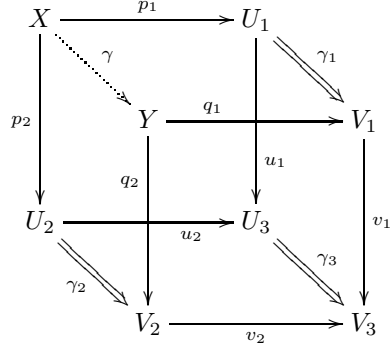
\[
      \xymatrix{
        {X}
          \ar@{->}[rr]^{p_1}
          \ar@{->}[dd]_{p_2}
          \ar@{.>}[dr]^{\gamma}
        & & {U_1}
          \ar@{->}[dd]^(0.7){u_1}
          \ar@{=>}[dr]^{\gamma_1}
        \\
        & {Y}
          \ar@{->}[rr]^(0.3){q_1}
          \ar@{->}[dd]_(0.3){q_2}
        & & {V_1}
          \ar@{->}[dd]^{v_1}
        \\
        {U_2}
          \ar@{->}[rr]_(0.7){u_2}
          \ar@{=>}[dr]_{\gamma_2}
        & & {U_3}
          \ar@{=>}[dr]^{\gamma_3}
        \\
        & {V_2}
          \ar@{->}[rr]_{v_2}
        & & {V_3}
      }
      \text{.}
    \]
  \item
    Similarly for initial objects and pushouts.
  \item
    Suppose we have two list universals for $L_i = \qeql(A_i)$ ($i = 1,2$)
    and an object equality $\gamma_A\colon A_1\Rightarrow A_2$,
    and an edge $\gamma_L\colon L_1\to L_2$ with sufficient data to characterize it as
    $\qeql(\gamma_A)$ (Remark~\ref{rem:listf}).
    Then $\gamma_L$ is an object equality.
  \end{enumerate}
\end{definition}

\begin{lemma}\label{lem:objeq}
  Let $\thT$ be a context.
  \begin{enumerate}
  \item
    If $\gamma\colon X\Rightarrow Y$ is an object equality, then in $\AUpres{\thT}$
    we have $X=Y$ and $\gamma$ is the identity morphism.
  \item
    If $\gamma \colon X \Rightarrow X$ is an object equality,
    then there is some equivalence extension $\thT\skeqext\thT'$ in which $\sknid(X)\skdiag\gamma$.
  \item
    If $\gamma \colon X \Rightarrow Y$ and $\gamma' \colon Y \Rightarrow Z$ are object equalities,
    then there is some $\thT\skeqext\thT'$ in which we have an object equality
    $\delta\colon X\Rightarrow Z$ and $\gamma\gamma'\skdiag\delta$.
  \item
    If $\gamma \colon X \Rightarrow Y$ is an object equality,
    then there is some $\thT\skeqext\thT'$ in which $\gamma$ is an isomorphism,
    and its inverse is also an object equality.
  \item
    If $\gamma,\gamma'\colon X\Rightarrow Y$ are two object equalities,
    then there is some $\thT\skeqext\thT'$ in which $\gamma\skdiag\gamma'$.
  \end{enumerate}
\end{lemma}
\begin{proof}
  (1) is immediate from the definition, bearing in mind that for a list universal
  the expression for $A\times L$ is defined to be that for the pullback of $!_L$ and $!_A$.

  (2) and (3) follow from the uniqueness clauses for fillins.

  (4) follows because all the cases for object equality are symmetric,
  and we can then apply (3) and (2).

  (5) again follows from the uniqueness clauses for fillins.
\end{proof}

We shall use the phrase ``object equality'' more generally than just for objects.

\begin{definition}\label{def:objEqGeneral}
  Let $\thT$ be a context.

  If $u_i\colon X_i \to  Y_i$ ($i=1,2$) are edges in $\thT$,
  then an \emph{object equality} from $u_1$ to $u_2$
  is the data of a commutative diagram
  \[
    \xymatrix{
      {X_1}
        \ar@{->}[r]^{u_1}_{\bullet}
        \ar@{->}[dr]^{\gamma_u}
        \ar@{=>}[d]_{\gamma_X}
      & {Y_1}
        \ar@{=>}[d]^{\gamma_Y}
      \\
      {X_2}
        \ar@{->}[r]_{u_2}^{\bullet}
      & {Y_2}
    }
  \]
  such that $\gamma_X$ and $\gamma_Y$ are object equalities.

  Let $f_0,f_1\colon\thT_1\skhom\thT_0$ be two context homomorphisms.
  Then an \emph{object equality} from $f_0$ to $f_1$ is a 2-cell $\gamma$ from
  $f_0$ to $f_1$, for which every carrier edge is an object equality
  $\gamma_X \colon f_0(X) \Rightarrow f_1(X)$.
  (It follows that for each edge $u\colon X \to Y$ of $\thT_1$,
  we get an object equality $(\gamma_X,\gamma_u,\gamma_Y)$ from $f_0(u)$ to $f_1(u)$.)

  By taking $\thT_1$ as either $\thob$ or $\thob^{\to}$,
  we see that object equality for homomorphisms subsumes the cases
  for nodes and edges.

  We say that two context homomorphisms are \emph{objectively equal} in $\thT$,
  symbolized $=_o$,
  if there is some equivalence extension of $\thT$ in which they have an object equality.
\end{definition}

\begin{proposition}\label{prop:objEqER}
  Objective equality of context homomorphisms is an equivalence relation.
\end{proposition}
\begin{proof}
  This is a straightforward extension of Lemma~\ref{lem:objeq}.
  For transitivity $f_0 =_o f_1 =_o f_2$,
  note that we may have different equivalence extensions for $f_0 =_o f_1$
  and for $f_1 =_o f_2$. Work in a common refinement.
\end{proof}

\section{Context maps}\label{sec:ConMaps}
In Section~\ref{sec:Con} we shall define a 2-category $\Con$ whose objects are contexts,
and whose morphisms $\thT_0 \to \thT_1$ are in bijection with strict AU-functors
$\AUpres{\thT_0} \leftarrow \AUpres{\thT_1}$.
In fact, its 1-cells will be what we shall define here as \emph{context maps}.

In this section we investigate the 1-category $\Con_{\skeqext\skhommap}$ of contexts and context maps,
from which $\Con$ is got by factoring out a congruence based on objective equality.
To save repetition, we shall exploit the fact that object equalities are a special case of 2-cells,
and the present section is really a collection of ad hoc preliminary results about
2-cells in the not-a-2-category $\Con_{\skeqext\skhommap}$.

We already have a category $\Con_{\skhom}$ of contexts and context homomorphisms --
and we shall also write $\Con_{\skhommap}$ for its opposite.
Recall that we consider two sketch homomorphisms equal if they agree on the nodes and edges.

\begin{definition}\label{def:contMorph}
  Let $\thT_0,\thT_1$ be contexts.
  Then a \emph{context map} from $\thT_0$ to $\thT_1$ is an opspan $(e,f)$
  from $\thT_0$ to $\thT_1$, where $e$ is an equivalence extension:
  \[
    \xymatrix@1{
      {\thT_0}
        \ar@{->}[r]^{e}_{\skeqext}
      & {\thT'_0}
      & {\thT_1}
        \ar@{->}[l]_{f}
    }
  \]
\end{definition}

Using reindexing, we can compose context maps.

\begin{definition}
  Suppose we have context maps as in the bottom two rows of the following diagram,
  and we reindex $e_1$ along $f_0$.
  \[
    \xymatrix{
      & & {\thT''_0}
      \\
      & {\thT'_0}
        \ar@{.>}[ur]^{f_0(e_1)}
      & & {\thT'_1}
        \ar@{.>}[ul]_{\varepsilon}
      \\
      {\thT_0}
        \ar@{->}[ur]^{e_0}_{\skeqext}
      & & {\thT_1}
        \ar@{->}[ul]_{f_0}
        \ar@{->}[ur]^{e_1}_{\skeqext}
      & & {\thT_2}
        \ar@{->}[ul]_{f_1}
    }
  \]
  Then the composite $(e_0,f_0)(e_1,f_1)$ is
  $(e_0 f_0(e_1), \varepsilon \circ f_1)$.
\end{definition}

Contexts and context maps form a category $\Con_{\skeqext\skhommap}$,
with composition as defined and identity maps $(\Id,\Id)$.
Note that $(e,f)$ is the composite $(e,\Id)(\Id,f)$.

\begin{definition}\label{def:ConTwoCell}
  A \emph{2-cell} in $\Con_{\skeqext\skhommap}$, between $\thT_0$ and $\thT_1$,
  is a context map $(e,\alpha)$ from $\thT_0$ to $\thT_1^{\to}$.
  Its domain and codomain are $(e,\alpha\circ i_{\lambda})$ ($\lambda = 0,1$).

  An \emph{object equality} is a 2-cell $(e,\gamma)$ in which $\gamma$ is an object equality.

  Two context maps $(e_i,f_i)$, with the same domain and codomain,
  are \emph{objectively equal} if $e_0$ and $e_1$ have a common refinement $e$
  such there is an object equality from $f_0\varepsilon_0$ to $f_1\varepsilon_1$.
  \begin{equation}\label{eq:mapObEq}
  \xymatrix{
    & {}
    \\
    {}
      \ar@{->}[ur]^{\varepsilon_0}
    & & {}
      \ar@{->}[ul]_{\varepsilon_1}
    \\
    & {}
      \ar@{->}[lu]^{e_0}
      \ar@{->}[ru]_(0.2){e_1}
      \ar@{->}[uu]^{e}
    & & {}
      \ar@{->}[lllu]^(0.2){f_0}
      \ar@{->}[lu]_{f_1}
  }
  \end{equation}
\end{definition}
From Proposition~\ref{prop:objEqER} it is easy to see that objective equality is an equivalence relation on each hom-set of $\Con_{\skeqext\skhommap}$.

$\Con_{\skeqext\skhommap}$ is not a 2-category --
it lacks vertical and horizontal composition.
For example, suppose we have two vertically composable 2-cells between $\thT_0$ and $\thT_1$.
To compose them we need to be able to compose the carrier edges in $\thT_0$.

For the time being we examine \emph{whiskering}, horizontal composition of 2-cells with 1-cells.

\emph{Left whiskering}%
\footnote{
  It is arguable which is left and which is right.
  We take it that left whiskering is for when the 1-cell is on the left
  in diagrammatic order of context maps.
}
is done by composition of context maps
$\xymatrix@1{{\thT_0} \ar@{->}[r] & {\thT_1} \ar@{->}[r] & {\thT_2^{\to}}}$.

\emph{Right whiskering} by context maps $(\Id,f)$ is similar,
with a composition
$\xymatrix@1{{\thT_0} \ar@{->}[r] & {\thT_1^{\to}} \ar@{->}[r]^{(\id,f^{\to})} & {\thT_2^{\to}}}$.

For whiskering as defined so far, it is clear that --
\begin{enumerate}
\item
  all possible associativities hold, and
\item
  whiskering preserves object equalities.
\end{enumerate}

The remaining case is right whiskering by maps $(e,\Id)$.
For these we start to need equivalence extensions.

\begin{lemma}\label{lem:eqExtTwoCell}
  Let $\thT_1$, $\thT'_1$ and $\thT_0$ be contexts.
  Suppose we have an equivalence extension $e_1 \colon \thT_1 \skeqext \thT'_1$,
  two homomorphisms $f_0,f_1\colon\thT'_1\skhom\thT_0$,
  and a 2-cell $\alpha\colon e_1 f_0 \to e_1 f_1$.
  Then,
  \begin{enumerate}
  \item
    There is some equivalence extension $e_0\colon\thT_0\skeqext\thT'_0$
    and a 2-cell $\alpha'\colon f_0 e_0 \to f_1 e_0$
    such that $\alpha e_0 = e_1^{\to} \alpha'$.
    \[
      \xymatrix{
        {\thT'_0}
        & {\thT'^{\to}_1}
          \ar@{.>}[l]_{\alpha'}
        & {\thT'_1}
          \ar@{->}[l]_{i_{\lambda}}
          \ar@{->}[dll]_(0.7){f_{\lambda}}
        \\
        {\thT_0}
          \ar@{.>}[u]^{e_0}
        & {\thT_1^{\to}}
          \ar@{->}[u]_(0.3){e_1^{\to}}
          \ar@{->}[l]^{\alpha}
        & {\thT_1}
          \ar@{->}[u]_{e_1}
          \ar@{->}[l]^{i_{\lambda}}
      }
    \]
  \item
    For any such $e_0$ and $\alpha'$ as in (1),
    suppose we also have (for the same $e_0$) $\alpha''$
    satisfying the same conditions as for $\alpha'$.
    Then $\alpha' =_{o} \alpha''$.
    (This just means that unary commutativities can be found
    between the actions of $\alpha'$ and $\alpha''$ on edges,
    since their actions on nodes are already constrained up to equality by $f_0$ and $f_1$.)
  \end{enumerate}
\end{lemma}
\begin{proof}
  It suffices to consider simple equivalence extensions $e_1$,
  and the only non-trivial ones are those that introduce nodes or edges.
  If $e_1$ introduces only commutativities,
  then the action of $\alpha'$ is already explicit in that of $\alpha$
  and $e_0$ just has to introduce the images under $f_0$ and $f_1$
  of those commutativities.
  This applies to the unit and associativity rules,
  and to the rules for the uniqueness of fillins.

  For the first case, suppose $e_1$ adjoins a composite $uv \skdiag_{XYZ} w$.
  In any case where $e_1$ introduces an edge $w\colon X\to Z$,
  in $\thT'^{\to}_1$ we have fresh edges $i_{\lambda}(w)$ and $\theta_w$,
  whose interpretations under $\alpha'$ must be $f_{\lambda}(w)$ and $\alpha_X f_1(w)$.
  This is already enough to prove the uniqueness, (2).
  For existence, we can certainly adjoin the composite $\alpha_X f_1(w)$ in an
  equivalence extension.
  Algebraically we check that appropriate square for $w$ commutes:
  \begin{align*}
    f_0(w)\alpha_Z & = f_0(u)f_0(v)\alpha_Z = f_0(u)\alpha_v = f_0(u)\alpha_Y f_1(v)\\
      & = \alpha_u f_1(v) = \alpha_X f_1(u) f_1(v) = \alpha_X f_1(w)
    \text{.}
  \end{align*}
  By Proposition~\ref{prop:eqExtEqLogic} we can find an equivalence extension with sufficient edges and commutativities to express this.

  A similar argument applies to all those equivalence extensions
  that adjoin an inverse to a particular edge $u\colon X \to Y$.
  We check
  \begin{align*}
    f_0(u^{-1})\alpha_X & = f_0(u^{-1})\alpha_X f_1(\sknid(X)) = f_0(u^{-1})\alpha_X f_1(u)f_1(u^{-1}) \\
      & = f_0(u^{-1})f_0(u)\alpha_Y f_1(u^{-1}) = f_0(\sknid(Y))\alpha_Y f_1(u^{-1})
        = \alpha_Y f_1(u^{-1}) \text{.}
  \end{align*}

  Next, suppose a node $X$ is introduced by a universal.
  The commutativities required for a homomorphism ensure that $\alpha'_X$ has to be the
  canonical fillin,
  and then the appropriate squares commute with respect to the structure edges
  to ensure that we have a homomorphism.

  Finally we consider fillins.

  We first look at pullbacks.
  These will show the method also for pushouts, terminals and initials,
  although list objects are more complicated.

  Suppose in $\thT_1$ we have a pullback $P$ of some opspan,
  and suppose that $u\colon Y \to P$
  fills in for a cone that has, for each projection $p\colon P\to X$,
  a morphism $q\colon Y \to X$.
  We need to show $f_0(u)\alpha_P \skdiag \alpha_Y f_1(u)$,
  and it suffices to show that when composed with each pullback projection for $f_1(p)$.
  \[
    \xymatrix@C=2cm{
      {f_0(Y)}
        \ar@{->}[r]^{\alpha_Y}
        \ar@{->}[d]_{f_0(u)}
        \ar@{->}@/_3pc/[dd]_{f_0(q)}
      & {f_1(Y)}
        \ar@{->}[d]^{f_1(u)}
        \ar@{->}@/^3pc/[dd]^{f_1(q)}
      \\
      {f_0(P)}
        \ar@{->}[d]_{f_0(p)}
        \ar@{->}[r]_{\alpha_P}
      & {f_1(P)}
        \ar@{->}[d]^{f_1(p)}
      \\
      {f_0(X)}
        \ar@{->}[r]_{\alpha_X}
      & {f_1(X)}
    }
  \]
  The bounding quadrangle, the lower small rectangle and the two side-bows all commute,
  and so (in some suitable equational extension) we can show
  $f_0(u)\alpha_P f_1(p) \skdiag \alpha_Y f_1(u)f_1(p)$.

  Finally we look at list fillins.
  Suppose $\thT_1$ has the data for a list fillin (see diagrams~\eqref{eq:eqExtLfill})
  and $\thT'_1$ adjoins the fillin $r$.
  Our task is to show that the two composites $f_0(r)\alpha_Y$ and $\alpha_{L\times B}f_1(r)$
  are equal in some equivalence extension,
  and it suffices to show that they are both fillins for
  \[
    \xymatrix@1{
      {f_0(B)}
        \ar@{->}[r]_{\alpha_B}
      & {f_1(B)}
        \ar@{->}[r]_{f_1(y)}
      & {f_1(Y)}
      & {f_1(A)\times f_1(Y)}
        \ar@{->}[l]^-{f_1(g)}
      && {f_0(A)\times f_1(Y)}
        \ar@{->}[ll]^-{\alpha_A \times f_1(Y)}
    }
  \]
  This is left to the reader.
\end{proof}

Note that if $\alpha$ is an object equality then so is $\alpha'$.
In other words, we can cancel equivalence extensions $e$ from objective equalities:
if there is an object equality from $ef_1$ to $ef_2$,
then $f_1$ and $f_2$ are objectively equal.

\begin{definition}\label{def:rtWhisker}
  Let $\alpha \colon \thT_1^{\to} \skhom \thT_0$ be a 2-cell,
  with domain and codomain
   $f_0$ and $f_1$,
  and let $e_1\colon\thT_1 \skeqext \thT'_1$ be an equivalence extension.
  Reindexing $e_1$ along $f_0$ and $f_1$ gives two equivalence extensions $f_{\lambda}(e_1)$
  of $\thT_0$,
  by homomorphisms $\varepsilon_i$.
  (See diagram~\eqref{eq:rtWhisk}.)

  Then a \emph{right whiskering} $(\Id,\alpha)(e_1,\Id)$ is
  a map $(e,\alpha')$ where $e$ is a common refinement of $f_0(e_1)$ and $f_1(e_1)$,
  and $\alpha'$ is a 2-cell from $\varepsilon_0$ to $\varepsilon_1$ in $e$.
\end{definition}
Note that, because of the need to use a common refinement of $f_0(e_1)$ and $f_1(e_1)$,
the domain of the whiskering is not strictly equal to what it should be at the 1-cell level.
However, they are objectively equal.
The codomain is similar.

\begin{proposition}\label{prop:rtWhisk}
  Right whiskering $(\Id,\alpha)(e_1,\Id)$ exists and is unique up to objective equality.
\end{proposition}
\begin{proof}
  First, reindex $f_1(e_1)$ along $f_0(e_1)$ to obtain a common refinement.
  We now apply Lemma~\ref{lem:eqExtTwoCell} to
  $\varepsilon_0 ; f_0(e)(f_1(e))$ and $\varepsilon_1 ; \varepsilon_{01}$
  to obtain an equivalence extension $e'$ and a 2-cell $\alpha'$ from
  $\varepsilon_0 ; f_0(e)(f_1(e)) ; e'$ to $\varepsilon_1 ; \varepsilon_{01} ; e'$.
  ($\thT_0$ in Lemma~\ref{lem:eqExtTwoCell} is $\thT'_0$ here.)

  The right whiskering $(\Id,\alpha)(e_1,\Id)$ is
  $(f_0(e_1);f_0(e_1)(f_1(e_1));e', \alpha')$.

  \begin{equation}\label{eq:rtWhisk}
  \xymatrix{
    & {}
    \\
    & {\thT'_0}
      \ar@{->}[u]^{e'}
    \\
    {}
      \ar@{->}[ur]^{f_0(e_1)(f_1(e_1))}
    & & {}
      \ar@{->}[ul]_{\varepsilon_{01}}
    & & {\thT'_1}
      \ar@{->}[ll]_{\varepsilon_1}
      \ar@{->}@/^0.5pc/[llll]^(0.8){\varepsilon_0}
      \xtwocell[llluu]{}^{f'_0}_{f'_1}{\alpha'}
    \\
    & {\thT_0}
      \ar@{->}[ul]^{f_0(e_1)}
      \ar@{->}[ur]_(0.6){f_1(e_1)}
    & & & {\thT_1}
      \xtwocell[lll]{}^{f_0}_{f_1}{\alpha}
      \ar@{->}[u]_{e_1}
  }
  \end{equation}

  The uniqueness part~(2) of Lemma~\ref{lem:eqExtTwoCell} now tells us
  that the 2-cell $\alpha'$ is unique up to unary commutativities of edges,
  so the right whiskering is unique up to objective equality.
\end{proof}

General right whiskering can now be defined by
\[
  (e_0,\alpha_0)(e_1,f_1) = (e_0,\Id)((\Id,\alpha)(e_1,\Id))(\Id,f_1)
  \text{.}
\]

\begin{proposition}\label{prop:whisker}
  \begin{enumerate}
  \item
    Whiskering obeys the usual associative laws up to objective equality.
  \item
    Whiskering preserves object equalities.
  \end{enumerate}
\end{proposition}
\begin{proof}
  (1)
  After what we said earlier, the only remaining issue is the associativity of
  $(\Id,\alpha)(e_0,\Id)(e_1,\Id)$.

  $((\Id,\alpha)(e_0,\Id))(e_1,\Id)$ has the property required for
  $(\Id,\alpha)(e_0 e_1,\Id)$, so they are objectively equal.

  (2)
  Clear from the remark after Lemma~\ref{lem:eqExtTwoCell}.
\end{proof}

Finally we prove the following lemma.
Note that if $cg$ is \emph{equal} to $f$,
then $e$ can be trivial, with $\varepsilon g' = g$.
With object equalities there is a little more work, and it is embodied in $e$.

\begin{lemma}\label{lem:Conpb}
  \begin{enumerate}
  \item
    Suppose we have the solid parts of the following diagram,
    \[ \xymatrix{
      & {\thT'_0}
        \ar@{.>}[dl]_{g'}
      & {\thT'_1}
        \ar@{->}[l]_{\varepsilon}
        \ar@{->}[dl]_{g}
      \\
      {\thT''_0}
      & {\thT_0}
        \ar@{.>}[l]^{e}
        \ar@{->}[u]^{f(c)}
      & {\thT_1}
        \ar@{->}[l]^{f}
        \ar@{->}[u]_{c}
    } \]
    where $c$ is an extension, the square is the reindexing,
    and we have an object equality $\gamma\colon f \Rightarrow cg$.

    Then we can find an equivalence extension $e\colon\thT_0\skeqext\thT''_0$
    and a homomorphism $g'\colon\thT'_0\to\thT''_0$ such that
    $f(c)g'$ is strictly equal to $e$
    and there is an object equality $\gamma'\colon \varepsilon g' \Rightarrow ge$
    such that $c\gamma' = \gamma e$.
  \item
    Suppose, in the situation above, we have an equivalence extension $e$
    and two homomorphisms $g'_i$ with the properties described.
    Then $g'_1$ and $g'_2$ are objectively equal in $\thT''_0$.
  \end{enumerate}
\end{lemma}
\begin{proof}
  (1) By induction we can assume that $c$ is a simple extension.

  If $c$ adjoins a primitive node $X$,
  then we define $e$ as trivial, and $g'(X) = g(X)$.

  If $c$ adjoins a primitive edge $u\colon X \to Y$ then in $\thT_0$ we have the solid part of
  \[ \xymatrix{
    {f(X)}
      \ar@{=>}[r]^{\gamma_X}
      \ar@{.>}[d]_{g'(u)}
    & {g(c(X))}
      \ar@{->}[d]^{g(u)}
    \\
    {f(Y)}
      \ar@{=>}[r]_{\gamma_Y}
    & {g(c(Y))}
  } \]
  and in a suitable equivalence extension of $\thT_0$
  we can define $g'(u)$ to make the square commute.

  Suppose $c$ adjoins a commutativity $vw\skdiag_{XYZ} u$.
  We have
  \[ \xymatrix{
    & {f(X)}
      \ar@{=>}[rr]
      \ar@{->}[dl]_{f(v)}
      \ar@{->}[dd]^(0.7){f(u)}
    && {g(c(X))}
      \ar@{->}[dl]_{g(c(v))}
      \ar@{->}[dd]^{g(c(u))}
    \\
    {f(Y)}
      \ar@{=>}[rr]
      \ar@{->}[dr]_{f(w)}
    && {g(c(Y))}
      \ar@{->}[dr]_{g(c(w))}
    \\
    & {f(Z)}
      \ar@{=>}[rr]
    && {g(c(Z))}
  } \]
  The square faces all commute because they are object equalities.
  Once $c$ has made the right-hand triangle commute,
  in a suitable equivalence extension we can deduce that so does the left-hand one.

  If $c$ adjoins a universal, then we let $e$ adjoin the same universal.

  (2) Every ingredient of $\thT'_0$ is in the image of either $f(c)$ or $\varepsilon$.
  It therefore suffices to note that $f(c)g'_1$ and $f(c)g'_2$ are strictly equal,
  while $\varepsilon g'_1$ and $\varepsilon g'_2$ are objectively equal
  by Proposition~\ref{prop:objEqER}.
\end{proof}

\section{The 2-category of contexts}\label{sec:Con}
We now define our 2-category $\Con$ in which the 0-cells are contexts,
and the 1-cells between $\thT_0$ and $\thT_1$ are in bijection with strict AU-functors
from $\AUpres{\thT_1}$ to $\AUpres{\thT_0}$.
At the same time, we shall make the reversal of direction by which a strict AU-functor
can be thought of as a transformation of models.
Thus we shall think of a 1-cell as a ``map'' from the ``space of models of $\thT_0$''
to the ``space of models of $\thT_1$''.

\subsection{$\Con$ as a 1-category}\label{sec:ConOneCat}
\begin{proposition}\label{prop:AUmapscat}
  Objective equality of context maps is a congruence on $\Con_{\skeqext\skhommap}$.

  Hence contexts and their maps modulo objective equality form a category $\Con$.
\end{proposition}
\begin{proof}
  It has already been remarked that objective equality is an equivalence relation on each hom-set.
  To show that it is a congruence, we show that if two context maps are objectively equal,
  then their composites with any $(e,f)$ are also objectively equal.
  On the left, we just reindex everything along $f$.
  On the right, we apply right whiskering by $(e,f)$,
  and use the fact that this preserves objective equality.
\end{proof}

We now have a functor $(\Id,-) \colon \Con_{\skhommap} \to \Con$ given by
    \[
      (\xymatrix@1{{\thT_0} \ar@{<-}[r]^{f} & {\thT_1}})
      \mapsto
      (\xymatrix@1{{\thT_0} \ar@{=}[r] & {\thT_0} \ar@{<-}[r]^{f} & {\thT_1}})
      \text{.}
    \]

\begin{theorem}\label{thm:ConUni}
  $\Con$ is free over $\Con_{\skhommap}$ subject to
  object equalities becoming equalities,
  and equivalence extensions becoming invertible.
\end{theorem}
\begin{proof}
  If $e\colon \thT_0 \skeqext \thT'_0$ is an equivalence extension,
  then $(\Id,e)$ has inverse $(e,\Id)$ in $\Con$.

  We have $(e,\Id);(\Id,e) = (e,e)$, and this is objectively equal to $(\Id,\Id)$
  using $e$ as a refinement of $\Id$.

  For the other composite we get $(e(e),\varepsilon)$ by reindexing.
  Now by the remark preceding Lemma~\ref{lem:Conpb}, with $g$ as an identity,
  we get a homomorphism $g'$ with $e(e);g' = \varepsilon;g' = \Id$,
  showing that $(e(e),\varepsilon)$ is equal to the identity.

  It follows that, in $\Con$, every morphism can be expressed in the form
  $(\Id,e)^{-1};(\Id,f)$, where $e$ is an equivalence extension.

  Now suppose we have a functor $F\colon \Con_{\skhommap} \to \catC$
  with those properties.
  We must show it factors uniquely via $(\Id,-)$, with $F'\colon \Con \to \catC$.
  Uniqueness is clear: we must have
  \[
    F'(e,f) = F'(e,\Id);F'(\Id,f) = F(e)^{-1};F(f)
    \text{.}
  \]

  For existence, first we show that $F'$ thus defined transforms objective equality
  to equality.
  Suppose $(e_i,f_i)$ ($i = 0,1$) are objectively equal,
  as in diagram~\eqref{eq:mapObEq}.
  Then
  \[
    F(e_i)^{-1};F(f_i) = F(e)^{-1};F(\varepsilon_i);F(f_i)
      = F(e)^{-1};F(f_i \varepsilon_i)
  \]
  and these are equal for $i=0,1$ because $F$ transforms object equality to equality.

  It is obvious that $F'$ preserves identities,
  and for composition it suffices to consider the composite $(\Id,f);(e,\Id) = (f(e),\varepsilon)$.
  In $\catC$ we have
  \[
    \begin{split}
      F(f);F(e)^{-1} & = F(f(e))^{-1};F(f f(e));F(e)^{-1}
          = F(f(e))^{-1};F(e\varepsilon);F(e)^{-1} \\
        & = F(f(e))^{-1};F(\varepsilon)
          \text{.}
    \end{split}
\]
\end{proof}

\begin{lemma}\label{lem:ConpbOneCat}
  \begin{enumerate}
  \item
    Any reindexing square~\eqref{eq:reindex} for a context extension becomes a pullback
    square in $\Con$.
  \item
    In $\Con$, extension maps (i.e. those of the form $(\Id,c)$ where $c$ is an extension)
    can be pulled back along any morphism.
  \end{enumerate}
\end{lemma}
\begin{proof}
  (1)
  Consider a diagram as on the left here, with the outer square commuting.
  \[
    \xymatrix{
      {\thU}
        \ar@{->}@/^1pc/[rr]^{(e_2,g_2)}
        \ar@{->}[dr]_{(e_1,g_1)}
      & {\thT'_0}
        \ar@{->}[r]_{(\Id,\varepsilon)}
        \ar@{->}[d]^{(\Id,f(c))}
      & {\thT'_1}
        \ar@{->}[d]^{(\Id,c)}
      \\
      & {\thT_0}
        \ar@{->}[r]_{(\Id,f)}
      & {\thT_1}
    }
    \quad
    \xymatrix{
      & {}
        \ar@{.>}[dl]_{g'_2}
      & {\thT'_0}
        \ar@{->}[l]_{\varepsilon'}
      & {\thT'_1}
        \ar@{->}[l]_{\varepsilon}
        \ar@{->}[dll]^(0.2){g_2}
      \\
      {}
      & {\thU}
        \ar@{.>}[l]^{e}
        \ar@{->}[u]_(0.6){g_1(f(c))}
      & {\thT_0}
        \ar@{->}[u]_(0.3){f(c)}
        \ar@{->}[l]^{g_1}
      & {\thT_1}
        \ar@{->}[u]_{c}
        \ar@{->}[l]^{f}
    }
  \]
  Taking a common refinement of $e_1$ and $e_2$,
  we might as well assume that they are both trivial and that we have an object equality
  $fg_1 \Rightarrow cg_2$.
  Now consider the diagram on the right,
  and apply Lemma~\ref{lem:Conpb} with $g_2$ for $g$.
  We obtain $e$ and $g_2'$, with $e$ an equivalence extension,
  strict equality $g_1(f(c));g'_2 = e$,
  and an object equality $\varepsilon\varepsilon' g'_2 \Rightarrow g_2 e$.

  The required fillin is $(e,\varepsilon' g'_2)$.
  It has the correct composites with $(\Id,f(c))$ and $(\Id,\varepsilon)$.
  Moreover, uniqueness follows by the same argument as in Lemma~\ref{lem:Conpb}.
  
  (2)
  After part (1), it suffices to show that $(\Id,c)$ can be pulled back
  along any map $(e,\Id)$ where $e\colon\thT_0 \skeqext \thT_1$ is an equivalence extension.
  This is trivial, because pullbacks along invertible morphisms always exist.
\end{proof}

\subsection{$\Con$ as 2-category}\label{sec:ConTwoCat}

We now develop the 2-categorical structure.

\begin{lemma}\label{lem:arrarr}
  Let $\thT$ be a context.
  Then $(\thT^{\to})^{\to}$ has an involution $(e,\tau)$
  such that $(\Id,i_{\mu})(e,\tau) = (\Id,i_{\mu}^{\to})$.
\end{lemma}
\begin{proof}
  We shall write $i_{\lambda\mu}$ ($\lambda,\mu = 0,1$) for the composite
  \[
    i_{\lambda}i_{\mu}=
    \xymatrix@1{
      {\thT}
        \ar@{->}[r]_{i_{\lambda}}
      & {\thT^{\to}}
        \ar@{->}[r]_{i_{\mu}}
      & {(\thT^{\to})^{\to}}
    }
    \text{.}
  \]
  In $(\thT^{\to})^{\to}$ we write $\theta$ for the first level homomorphism, in $\thT^{\to}$,
  represented in $(\thT^{\to})^{\to}$ by $i_{\mu}(\theta)$,
  and $\phi$ for the second level homomorphism.

  Note that $i_{\lambda}i_{\mu}^{\to} = i_{\mu}i_{\lambda}$.
  It follows that any model of $(\thT^{\to})^{\to}$ has a square of
  four models of $\thT$, got from the $i_{\lambda\mu}$s,
  and four homomorphisms between them, got from the $i_{\mu}$s and the $i_{\mu}^{\to}$s.
  In fact, the square will commute, because $\phi$ is homomorphic with respect to the $i_{\mu}$s.
  Conversely, any such commutative square of homomorphisms gives a model of $(\thT^{\to})^{\to}$.
  \[
    \xymatrix{
      {i_{00}(X)}
        \ar@{->}[r]^{i_0(\theta_X)}
        \ar@{->}[d]_{\phi_{i_0(X)} = i_0^{\to}(\theta_X)}
      & {i_{10}(X)}
        \ar@{->}[d]^{\phi_{i_0(X)} = i_1^{\to}(\theta_X)}
      \\
      {i_{01}(X)}
        \ar@{->}[r]_{i_1(\theta_X)}
      & {i_{11}(X)}
    }
  \]
  Reflecting the square about its leading diagonal gives another such square,
  and that is the essential action of $\tau$.
  The only remaining issue is that in the \emph{context} $(\thT^{\to})^{\to}$,
  we need an equivalence extension to introduce some composites and associativities --
  mere commutativity of the squares (of carrier edges) does not explicitly
  have all the data for a homomorphism between homomorphisms.
\end{proof}

\begin{lemma}\label{lem:homfunctori}
  Let $f_0,f_1\colon\thT_1 \skhom \thT_0$ have an object equality $\gamma$.
  Then $f_0^{\to}$ and $f_1^{\to}$ are objectively equal.
\end{lemma}
\begin{proof}
  Use $(\Id,\gamma^{\to})(e,\tau)$, where $(e,\tau)$ is as in Lemma~\ref{lem:arrarr}.
\end{proof}

\begin{lemma}\label{lem:homfunctorii}
  Let $e \colon \thT_1 \skeqext \thT_0$ be an equivalence extension.
  Then $e^{\to}$ is invertible in $\Con$.
\end{lemma}
\begin{proof}
  The identity on $\thT_1^{\to}$ gives the generic 2-cell between $\thT_1^{\to}$ and $\thT_1$,
  its domain and codomain being $i_0$ and $i_1$.
  Consider its right whiskering (Definition~\ref{def:rtWhisker}) by $(e,\Id)$, giving
  \[
    \xymatrix@1{
      {\thT_1^{\to}} \ar@{->}[r]^{e'}_{\skeqext}
      & {\thU}
      & {\thT_0^{\to}} \ar@{->}[l]_{\alpha}
    }
    \text{.}
  \]
  Then $(e',\alpha)$ is the inverse of $(\Id,e^{\to})$.

  First, $(e',\alpha)(\Id,e^{\to}) = (e', e^{\to}\alpha) = (e',e')$.

  Next, for $(\Id,e^{\to})(e',\alpha)$ consider
  \[
    \xymatrix{
      {}
      & {\thU}
        \ar@{->}[l]_{\varepsilon}
      & {\thT_0^{\to}}
        \ar@{->}[l]_{\alpha}
      \\
      {\thT_0^{\to}}
        \ar@{->}[u]^{e^{\to}(e')}
      & {\thT_1^{\to}}
        \ar@{->}[l]^{e^{\to}}
        \ar@{->}[u]^{e'}
      & {\thT_1^{\to}}
        \ar@{=}[l]{}
        \ar@{->}[u]_{e^{\to}}
    }
  \]
  $(e^{\to}(e'),\alpha \varepsilon)$ is a right whiskering of $(\Id,e^{\to})$ by $(e,\Id)$;
  but then so is $(e^{\to}(e'),e^{\to}(e'))$, and so they are objectively equal,
  and the latter is objectively equal to the identity on $\thT_0^{\to}$.
\end{proof}

\begin{theorem}\label{thm:homfunctor}
  The functor $-^{\to}$ on $\Con_{\skhom}$ gives an endofunctor on $\Con$.
\end{theorem}
\begin{proof}
  Theorem~\ref{thm:ConUni} reduces this to Lemmas~\ref{lem:homfunctori} and~\ref{lem:homfunctorii}.
\end{proof}

We now define an internal category in the functor category $[\Con,\Con]$ in which
the object of objects is $\Id$, and the object of morphisms is $-^{\to}$.

The structure operations will be natural transformations.
Note that to prove naturality, it suffices to prove it with respect to maps of
the form $(\Id,f)$,
since the rest follows from invertibility of $(e,\Id)$.

The domain and codomain, natural transformations from $-^{\to}$ to $\Id$,
are given by the maps $\dom = (\Id,i_0)$ and $\cod = (\Id,i_1)$.

The identity $\Id\colon\Id\to -^{\to}$ is given by maps $(e,\gamma)$
where $\gamma\colon\thT^{\to}\skhom\thT'$ takes $\theta$ to the equality homomorphism
on the generic model of $\thT$.
The equivalence extension $e\colon\thT\skeqext\thT'$ uses instances of the unit laws
to provide the necessary commutativities.

Since $i_0$ is an extension, we can reindex along $i_1$, and in fact this gives $\thT^{\to\to}$
as a pullback in $\Con$.
\[
  \xymatrix{
    {\thT^{\to\to}}
      \ar@{->}[r]^{(\Id,\varepsilon)}
      \ar@{->}[d]_{(\Id,i_1(i_0)}
    & {\thT^{\to}}
      \ar@{->}[d]^{\dom = (\Id,i_0)}
    \\
     {\thT^{\to}}
      \ar@{->}[r]_{\cod = (\Id,i_1)}
    & {\thT}
  }
\]
$i_1(i_0)$ maps the ingredients of $\thT^{\to}$ to the 0- and 1-copies in $\thT^{\to\to}$,
and adjoins the 2-copies with the carriers from 1 to 2.

In an equivalence extension of $\thT^{\to\to}$, the two model homomorphisms can be composed,
and this provides composition as a natural transformation from $-^{\to\to}$ to $-^{\to}$.
It is vertical composition of the two 2-cells $\thT^{\to}\skhom\thT^{\to\to}$.

Thus for each $\thT$ we get an internal category $N(\thT)$ in $\Con$,
on objects $\thT$ and morphisms $\thT^{\to}$.

Using the category structure of $N(\thT_1)$, this makes $\Con(\thT_0,\thT_1)$
into a category, with objects and morphisms the 1-cells and 2-cells between
$\thT_0$ and $\thT_1$.

We already have vertical composition of 2-cells.
(We shall compose from top to bottom, so the codomain of the upper 2-cell
must equal the domain of the lower.)

We deal with horizontal composition by whiskering.
Using the functor $-^{\to}$, we can make $\Con(-,-^{\to})$ into
a profunctor from $\Con$ to $\Con$, and this provides whiskering on both sides.
The proof of Lemma~\ref{lem:homfunctorii} shows that this agrees with the whiskering we
already have.

Horizontal composition can now be defined as
\[
  \alpha\beta
    = \frac{\alpha\dom(\beta)}{\cod(\alpha)\beta}
  \text{.}
\]
The interchange law follows from --
\begin{lemma}
  \[
    \frac{\alpha\dom(\beta)}{\cod(\alpha)\beta}
      = \frac{\dom(\alpha)\beta}{\alpha\cod(\beta)}
    \text{.}
  \]
\end{lemma}
\begin{proof}
  Suppose we have the following.
  \[
    \xymatrix{
      {\thT_0}
        \xtwocell[rr]{}^{}_{}{\quad(e,\alpha)}
      && {\thT_1}
        \xtwocell[rr]{}^{}_{}{\quad(e',\beta)}
      && {\thT_2}\text{.}
    }
  \]

  By whiskering $(e,\alpha)(e',\Id)$ we might as well assume that $e$ and $e'$
  are both identities.
  In $\Con_{\skhom}$ we now have
  \[
    \alpha\circ\beta^{\to} \colon
    \xymatrix@1{
      {\thT_0}
      & {\thT_1^{\to}}
        \ar@{->}[l]_{\alpha}
      & {(\thT_2^{\to})^{\to}}
        \ar@{->}[l]_{\beta^{\to}}
    }
  \]
  The two vertical composites in the statement are the images in $\thT_0$
  of the two routes round the square of homomorphisms in $(\thT_2^{\to})^{\to}$
  (see Lemma~\ref{lem:arrarr}) and so are equal.
\end{proof}

Putting together the properties proved so far,
we can deduce --

\begin{theorem}\label{thm:ConTwoCat}
  $\Con$ is a 2-category.
\end{theorem}

\subsection{Limits in $\Con$}\label{sec:ConLimits}
We have two main results here.
The first (Theorem~\ref{thm:PIE}) is that $\Con$ has finite PIE-limits
(products, inserters, equifiers~\cite{PowerRob:PIE}).

This is a large class of finite weighted limits, but a notable lack is equalizers and pullbacks.
Although by universal algebra $\AU_s$ has all pushouts and $\AU_s^{op}$ has all pullbacks,
in general we cannot replicate this in contexts.
For example, suppose we have two context homomorphisms $f_i\colon\thT_0\skhom\thT_i$
where $\thT_0$ has just a single node, and the $f_i$s map it to nodes introduced by two different
kinds of universals.
Then the pushout must specify an equality between those two different nodes,
and that cannot be done with a context.

The second main result (Theorem~\ref{thm:Conpb}) is that, nonetheless,
pullbacks of extension maps do exist, essentially by reindexing.
In fact this has already been addressed in Lemma~\ref{lem:ConpbOneCat}.
All that remains here is to show that they are 2-categorical conical limits
(in other words, they take proper account of 2-cells between fillins).

Note that all our weighted limits are \emph{strict}, with strict cones, as in~\cite{PowerRob:PIE}.
We do not follow the convention in~\cite[p.244]{Elephant1} of interpreting them in a ``pseudo'' sense.

Also note that we do not claim to have constructed the limits in a canonical way,
at least not those -- such as pullbacks, inserters and equifiers -- that depend on maps.
This is because the construction will depend on the representatives $(e,f)$ of the maps.

\subsubsection*{Pullbacks and products}

\begin{lemma}\label{lem:extArrow}
  Consider a context reindexing square~\eqref{eq:reindex}.
  Then the following square becomes a pullback in $\Con$.
  \begin{equation}\label{eq:extArr}
    \xymatrix{
      {\thT'^{\to}_0}
      & {\thT'^{\to}_1}
        \ar@{->}[l]_{\varepsilon^{\to}}
      \\
      {\thT_0^{\to}}
        \ar@{->}[u]^{f(c)^{\to}}
      & {\thT_1^{\to}}
        \ar@{->}[u]_{c^{\to}}
        \ar@{->}[l]^{f^{\to}}
    }
  \end{equation}
\end{lemma}
\begin{proof}
  If $c^{\to}$ were an extension, then we could apply Lemma~\ref{lem:ConpbOneCat}.
  In fact it is not, but only for bureaucratic reasons
  based on the concrete definition of coproduct ``+'' (see Section~\ref{sec:ConcreteAUT}).
  The issue is that the steps constructing $\thT'^{\to}_1$ are applied in an order
  that does not start off with all those for $\thT_1^{\to}$.
  Those steps can be reordered to give an extension $c'\colon\thT_1^{\to}\skext\thT''_1$
  isomorphic to $c^{\to}$,
  and moreover that reordering can be reindexed along $f^{\to}$ to get a reindexing square
  isomorphic to~\eqref{eq:extArr}:
  \[
    \xymatrix{
      {\thT'^{\to}_0}
      &&& {\thT'^{\to}_1}
        \ar@{->}[lll]_{\varepsilon^{\to}}
      \\
      & {\thT''_0}
        \ar@{->}[ul]^{\cong}
      & {\thT''_1}
        \ar@{->}[l]_{\varepsilon'}
        \ar@{->}[ur]_{\cong}
      \\
      {\thT_0^{\to}}
        \ar@{->}[uu]^{f(c)^{\to}}
        \ar@{->}[ur]_{f^{\to}(c')}
      &&& {\thT_1^{\to}}
        \ar@{->}[uu]_{c^{\to}}
        \ar@{->}[ul]^{c'}
        \ar@{->}[lll]^{f^{\to}}
    }
  \]
  By Lemma~\ref{lem:ConpbOneCat} the reindexing square is a pullback in $\Con$,
  and it follows that so too is~\eqref{eq:extArr}.
\end{proof}

\begin{theorem}\label{thm:Conpb}
  $\Con$ has pullbacks of extension maps along any map.
\end{theorem}
\begin{proof}
  Lemma~\ref{lem:ConpbOneCat} has already shown the 1-categorical form of this.
  It remains to show that we also have 2-cell fillins,
  and the ability to do this follows from Lemma~\ref{lem:extArrow}.
\end{proof}

\begin{lemma}\label{lem:ConFinProd}
  $\Con$ has all finite products.
\end{lemma}
\begin{proof}
  The empty theory $\thone$ is initial in $\Con_{\skhom}$.
  After that one easily shows that it is terminal in $\Con$.

  The case for binary products follows from Theorem~\ref{thm:Conpb},
  since the unique homomorphism $\thone\skhom\thT$ is an extension.
\end{proof}

\subsubsection*{Inserters}
\newcommand{\skins}{\mathsf{Ins}}     

First, we work in $\Con_{\skhom}$ (or, dually, in $\Con_{\skhommap}$).
\begin{definition}\label{def:Insert}
  Let $f_{\lambda}\colon \thT_1 \skhom \thT_0$ ($\lambda=0,1$) be two context homomorphisms.
  Then we define an extension $c\colon\thT_0 \skext\skins(f_0,f_1)$ by adjoining:
  \begin{itemize}
  \item
    for every node $Y$ in $\thT_1$, an edge $\theta_Y\colon f_0(Y)\to f_1(Y)$; and
  \item
    for every edge $u\colon Y \to Y'$ in $\thT_1$, an edge $\theta_u$ and two commutativities
    \[
      \xymatrix{
        {f_0(Y)}
          \ar@{->}[r]^{\theta_Y}_{\bullet}
          \ar@{->}[dr]^{\theta_u}
          \ar@{->}[d]_{f_0(u)}
        & {f_1(Y)}
          \ar@{->}[d]^{f_1(u)}
        \\
        {f_0(Y')}
          \ar@{->}[r]_{\theta_{Y'}}^{\bullet}
        & {f_1(Y')}
      }
    \]
  \end{itemize}
\end{definition}
Obviously this generalizes the construction of $\thT^{\to}$ out of $\thT^2$.
We have two strictly commutative squares
\[
  \xymatrix{
    {\skins(f_0,f_1)}
    & {\thT_1^{\to}}
      \ar@{->}[l]_{\overline{f}}
    \\
    {\thT_0}
      \ar@{->}[u]^{c}
    & {\thT_1}
      \ar@{->}[l]^{f_{\lambda}}
      \ar@{->}[u]_{i_{\lambda}}
  }
\]
and in fact $\skins(f_0,f_1)$ is their joint pushout in $\Con_{\skhom}$.

To put this another way, left whiskering induces a bijection between
\begin{enumerate}
\item
  context homomorphisms $g'\colon\skins(f_0,f_1)\skhom\thU$, and
\item
  pairs $(g,\theta)$ where $g\colon\thT_0\skhom\thU$ is a context homomorphism,
  and $\theta\colon f_0 g \to f_1 g$ is a 2-cell.
\end{enumerate}

This very nearly also works at the level of 2-cells.
Consider two sketch homomorphisms $g'_{\mu}\colon\skins(f_0,f_1)\skhom\thU$
($\mu=0,1$),
corresponding to pairs $(g_{\mu}=cg'_{\mu},\theta_{\mu})$ as above,
and suppose we have a 2-cell $\alpha'\colon g'_0 \to g'_1$.
Considering the nodes and edges of $\skins(f_0,f_1)$,
we see that the edge data needed for $\alpha'$ comprises edges of the form
$\alpha'_{cX}$ and $\alpha'_{cu}$, for nodes and edges in $\thT_0$, and
$\alpha'_{\theta_{Y}}$ and $\alpha'_{\theta_{v}}$, for nodes and edges in $\thT_1$.
The first two kinds come along with commutativitites that make the whiskered 2-cell
$g_0 \to g_1$.
The last two kinds have commutativities
\[
  \xymatrix{
    {g'_0 c f_0 Y}
      \ar@{->}[r]^{\alpha'_{c f_0 Y}}_{\bullet}
      \ar@{->}[d]_{g'_0 \theta_{Y}}
      \ar@{->}[dr]^{\alpha'_{\theta_{Y}}}
    & {g'_1 c f_0 Y}
      \ar@{->}[d]^{g'_1 \theta_{Y}}
    \\
    {g'_0 c f_1 Y}
      \ar@{->}[r]_{\alpha'_{c f_1 Y}}^{\bullet}
    & {g'_1 c f_1 Y}
  }
  \quad
  \xymatrix{
    {g'_0 c f_0 Y}
      \ar@{->}[r]^{\alpha'_{c f_0 Y}}_{\bullet}
      \ar@{->}[d]_{g'_0 \theta_{v}}
      \ar@{->}[dr]^{\alpha'_{\theta_{v}}}
    & {g'_1 c f_0 Y}
      \ar@{->}[d]^{g'_1 \theta_{v}}
    \\
    {g'_0 c f_1 Y'}
      \ar@{->}[r]_{\alpha'_{c f_1 Y'}}^{\bullet}
    & {g'_1 c f_1 Y'}
  }
  \text{.}
\]
The first of these expresses that the $\alpha'_{\theta_{Y}}$s give the correct carrier
edges for the horizontal composition of $\alpha'$ and $\theta$.
The second is equivalent to saying that the $\alpha'_{\theta_{v}}$s give the correct
naturality diagonals for this horizontal composition,
in other words
\[
  \xymatrix{
    {g'_0 c f_0 Y}
      \ar@{->}[r]^{\alpha'_{\theta_{Y}}}_{\bullet}
      \ar@{->}[dr]^{\alpha'_{\theta_{v}}}
      \ar@{->}[d]_{g'_0 c f_0 v}
    & {g'_1 c f_1 Y}
      \ar@{->}[d]^{g'_1 c f_1 v}
    \\
    {g'_0 c f_0 Y'}
      \ar@{->}[r]_{\alpha'_{\theta_{Y}}}^{\bullet}
    & {g'_1 c f_1 Y'}
  }
  \text{,}
\]
but only modulo applications of associativity laws.

\begin{lemma}\label{lem:ConInsert}
  $\Con$ has inserters.
\end{lemma}
\begin{proof}
  Suppose we have two maps from $\thT_0$ to $\thT_1$.
  We can represent them as homomorphisms into a single equivalence extension $\thT'_0$ of $\thT_0$.
  We shall show that $\thT_0\skeqext\thT'_0\skext\skins(f_0,f_1)$
  provides the inserter in $\Con$.

  In the following diagram we use arrows
  $\xymatrix@1{
    {}\ar@{(-}[r]^{(e,\Id)}
    &{,} \ar@{->}[r]^{(\Id,f)}
    &{,} \ar@{(->}[r]^{(e,f)} & {}
  }$
  for maps of the forms indicated.
  \[
    \xymatrix{
      {\thU'}
        \ar@{->}@/^1pc/[rr]
        \ar@{.>}[r]
      & {\skins(f_0,f_1)}
        \ar@{->}[r]_{c}
      & {\thT'_0}
        \ar@{->}[r]^{f_{\lambda}}
      & {\thT_1}
      \\
      {\thU}
        \ar@{(->}[rr]
        \ar@{(-}[u]
      & & {\thT_0}
        \ar@{(-}[u]_{e}
    }
  \]

  The map from $\skins(f_0,f_1)$ to $\thT_0$ is got by inverting the equivalence extension
  $e\colon\thT_0\skeqext\thT'_0$.

  Suppose we have a map from $\thU$ to $\thT_0$ and a 2-cell between its  composites with the $f_{\lambda}$s.
  By replacing $\thU$ by a suitable equivalence extension $\thU'$,
  we may assume that the 2-cell, between maps from $\thU'$ to $\thT_1$,
  is entirely in $\Con_{\skhommap}$ as in the above diagram,
  and we get a unique factorization $\thU' \to \skins(f_0,f_1)$ in $\Con_{\skhommap}$.
  This then gives us a unique factorization in $\Con$.

  The remarks before the lemma now enable us to extend this to 2-cells in the manner required for
  a weighted limit.
  (Now we need an equivalence extension of $\thU'$ for the associativities needed.)
\end{proof}

\subsubsection*{Equifiers}
\newcommand{\skeqf}{\mathsf{Eq}}     

Again, we start off in $\Con_{\skhommap}$.

\begin{definition}\label{def:Eq}
  Suppose we have two homomorphisms $\alpha,\beta\colon \thT_1^{\to} \skhom \thT_0$
  that, as 2-cells, have the same domain and codomain --
  $f_{\lambda} = i_{\lambda}\alpha = i_{\lambda} \beta$ ($\lambda=0,1$).
  (Equality is in the sense of agreeing on nodes and edges.)
  Then we define an extension $c\colon\thT_0\skext\skeqf(\alpha,\beta)$
  that adjoins unary commutativities $\alpha_Y\skdiag\beta_Y$ and $\alpha_v\skdiag\beta_v$
  for the nodes $Y$ and edges $v$ in $\thT_1$.
\end{definition}

A homomorphism $g'\colon\skeqf(\alpha,\beta)\skhom\thU$ is equivalent to a homomorphism
$g\colon\thT_0\skhom\thU$ such that $\alpha g$ and $\beta g$ are equal in the sense that
there are unary commutativities in $\thU$ equating the images under $g$ of the
$\theta_Y$s and the $\theta_v$s.

We can extend this precisely to 2-cells in $\Con_{\skhommap}$.
If $g'_{\lambda}$ are two homomorphisms from $\skeqf(\alpha,\beta)$ to $\thU$,
then a 2-cell from $g'_0$ to $g'_1$ is equivalent to a 2-cell from $cg'_0$ to $cg'_1$.

\begin{lemma}\label{lem:ConEquif}
  $\Con$ has equifiers.
\end{lemma}
\begin{proof}
  Suppose in $\Con$ we have two 2-cells between $\thT_0$ and $\thT_1$
  with equal domain and codomain.
  Then by taking common refinements,
  and vertically composing one of the 2-cells with object equalities,
  we can suppose without loss of generality that our 2-cells are given by
  a suitable equivalence extension $e\colon \thT_0\skeqext\thT'_0$
  and, entirely in $\Con_{\skhommap}$,
  two 2-cells between $\thT'_0$ and $\thT_1$ with equal domain and codomain.
  Then $\skeqf(\alpha,\beta)$, mapped through to $\thT_0$ using $(\Id,e)$,
  provides the equifier we seek.

  \[
    \xymatrix{
      {\skeqf(\alpha,\beta)}
        \ar@{->}[r]^-{c}
      & {\thT'_0}
        \ar@{->}@/^1pc/[rr]^{f_0}
        \ar@{}[rr]|{\Downarrow\alpha\Downarrow\beta}
        \ar@{->}@/_1pc/[rr]_{f_1}
      & & {\thT_1}
      \\
      & {\thT_0}
        \ar@{(-}[u]_{e}
    }
  \]
  The rest is similar to Lemma~\ref{lem:ConInsert}.
\end{proof}

\begin{theorem}\label{thm:PIE}
  $\Con$ has finite pie limits.
\end{theorem}
\begin{proof}
  This is the combined content of Lemmas~\ref{lem:ConFinProd},~\ref{lem:ConInsert}
  and~\ref{lem:ConEquif}.
\end{proof}

\section{A concrete construction of $\AUpres{\thT}$}\label{sec:ConcreteAUT}
We can define a 2-functor $\AUpres{-}\colon\Con\to\AU^{op}_s$,
acting on objects as $\thT\mapsto\AUpres{\thT}$.
(At the 1-category level this is immediate from Theorem~\ref{thm:ConUni},
using Proposition~\ref{prop:eqExtAUiso} and Lemma~\ref{lem:objeq}.)

The main result of this section, Theorem~\ref{thm:AUpresByEqExts},
is that this 2-functor is representable,
with $\AUpres{\thT}$ isomorphic to $\Con(\thT,\thob)$.
We also show, Theorem~\ref{thm:AUpres}, that it is full and faithful:
thus \emph{all} strict AU-functors between AUs of the form $\AUpres{\thT}$,
with $\thT$ a context, can be got by the finitary means of constructions in $\Con$.

Finally we shall show how the construction itself can be conducted entirely within
the logic of AUs.
This is in the spirit of the idea that AU constructions should be internalizable within AUs,
the idea that inspired Joyal's original use of them with regard to G\"odel's Theorem.

For the 2-cells, first note that $\AUpres{\thT^{\to}}$ is a tensor
$\mathbf{2}\otimes\AUpres{\thT}$ in $\AU_s$.
This is because a strict AU-functor $\AUpres{\thT^{\to}}\to\catA$
is equivalent to a strict model of $\thT^{\to}$ in $\catA$,
which is equivalent to a strict model of $\thT$ in $\catA\downarrow\catA$,
which is equivalent to a strict AU-functor $\AUpres{\thT} \to \catA\downarrow\catA$,
which is equivalent to a 2-cell between $\AUpres{\thT}$ and $\catA$
with domain and codomain both strict.

Hence $\AUpres{\thT^{\to}}$ is a cotensor $\mathbf{2}\pitchfork\AUpres{\thT}$
in $\AU_s^{op}$.
Thus we find that 2-cells in $\Con$, which are 1-cells to some $\thT^{\to}$,
are mapped to 2-cells in $\AU_s^{op}$,
and this preserves vertical and horizontal composition.

We next investigate the categories $\Con(\thT,\thob)$.
The objects and morphisms of this are the nodes and edges of equivalence extensions
of $\thT$, all modulo objective equality.

\begin{theorem}\label{thm:AUpresByEqExts}
  Let $\thT$ be a context.
  Then $\Con(\thT,\thob)$ is an AU freely presented by $\thT$,
  in other words $\AUpres{\thT} \cong \Con(\thT,\thob)$.
\end{theorem}
\begin{proof}
  All the AU constructions can be captured by equivalence extensions,
  and have the necessary properties.
  The rules of object equalities (for nodes) and fillin uniqueness (for edges)
  ensure that the constructions yield equals when applied to equals,
  and so have canonical representatives.
  Thus $\Con(\thT,\thob)$ is an AU.

  If $M$ is a strict model of $\thT$ in $\catA$,
  then any object or morphism in $\Con(\thT,\thob)$ gets a unique interpretation in $\catA$
  by model extension along the equivalence extension used.
  This respects objective equality,
  and so yields a well defined interpretation of the object or morphism.
\end{proof}

\begin{proposition}\label{prop:modelsAsProtosheaves}
  Let $\thT, \thT_0$ be contexts.
  If $(e,f)$ is a context map from $\thT$ to $\thT_0$,
  then the nodes and edges of $\thT_0$, translated along $f$,
  give a strict model of $\thT_0$ in $\Con(\thT,\thob)$.
  This induces a bijection between
  \begin{itemize}
  \item
    context maps from $\thT$ to $\thT_0$ (modulo objective equality), and
  \item
    strict models of $\thT_0$ in $\Con(\thT,\thob)$.
  \end{itemize}
\end{proposition}
\begin{proof}
  Objective equality of the context maps is determined solely by objective equalities
  for their nodes and edges, which is equality of the models in $\Con(\thT,\thob)$.
  Hence we have injectivity.

  For surjectivity, each piece of data for a strict model of $\thT_0$ is expressed in
  an equivalence extension of $\thT$.
  There are only finitely many of these, so they have a common refinement $e$, say,
  and then the strict model can be expressed as a context map $(e,f)$.
\end{proof}

\begin{theorem}\label{thm:AUpres}
  The 2-functor $\AUpres{-}$ is full and faithful on 1-cells and 2-cells.
\end{theorem}
\begin{proof}
  Let $\thT_0$ and $\thT_1$ be contexts.
  Strict AU-functors $\AUpres{\thT_1} \to \AUpres{\thT_0}$ are equivalent to
  strict models of $\thT_1$ in $\AUpres{\thT_0} \cong \Con(\thT,\thob)$,
  and these are equivalent to 1-cells in $\Con$.

  The result for 2-cells follows by considering maps to arrow contexts $\thT_1^{\to}$.
\end{proof}

We now look at the concrete construction in AU logic.

Each kind $\sigma$ of simple extension or simple equivalence extension
takes some given \emph{data}, and produces a \emph{delta}.
The possible data are given by a functor $\extdat_{\sigma}$ from sketches to sets.
More carefully, an element of $\extdat_{\sigma}(\thT)$
is some finite tuple of elements of carriers in $\thT$,
subject to some equations.
Hence $\extdat_{\sigma}$ can be understood as an object of the cartesian classifying category
for the unary theory of sketches,
and for any sketch $\thT$ in a cartesian category $\catC$,
$\extdat_{\sigma}(\thT)$ is an object of $\catC$.
If the sketch $\thT$ is in an AU, then, for each element of $\extdat_{\sigma}(\thT)$,
the delta now gives us a proto-extension $\thT\skhom\thT'$.

Since there are only finitely many kinds of simple extension or simple equivalence extension,
in an AU we can sum over them and get
\[ \begin{split}
  \extdatse & \triangleq
    \sum \{ \extdat_{\sigma} \mid \sigma \text{ a kind of simple extension} \} \text{,}\\
  \extdatsee & \triangleq
    \sum \{ \extdat_{\sigma} \mid \sigma \text{ a kind of simple equivalence extension} \text{.} \}
\end{split} \]

Let us now restrict ourselves to strongly finite sketches,
in other words, sketches in the category $\Fin$ whose objects are natural numbers
and whose morphisms are functions between the corresponding finite cardinals.
This can be defined internally in any AU.
We obtain an internal graph $\Sk_{s\skext}$ whose nodes are strongly finite sketches $\thT$,
and whose edges are pairs $(\thT, e\in\extdatse(\thT))$ --
the source is $\thT$, the target is the corresponding simple extension $\thT'$.
Note that we can, and shall, choose the deltas in such a way that, for every carrier,
the corresponding carrier function for the extension is the natural inclusion
for some natural numbers $m\leq n$.
We write $\Sk_{\skext}$ for the path category of $\Sk_{s\skext}$,
its morphisms being the composable tuples of edges.
(Note that two different paths could still give the same extension.)

We can now take the contexts to be the targets of extensions
whose domains are the empty sketch $\thone$.

Next we do the same with equivalence extensions,
to obtain a graph $\Sk_{s\skeqext}$ and its path category $\Sk_{\skeqext}$.

Note that if $f\colon\thT_1 \skhom \thT_2$ then $f$ extends to a function
$\extdatse(\thT_1)\to\extdatse(\thT_2)$,
and so transforms any extension $c$ of $\thT_1$ into one $f(c)$ of $\thT_2$.
This is the reindexing, and it applies similarly to equivalence extensions.

From these ingredients we can now, internally in any AU,
define the 2-category $\Con$ and also, from any internal context $\thT$,
define $\Con(\thT,\thob)$ and hence $\AUpres{\thT}$.

\section{Conclusion}\label{sec:Conc}
The present paper has fulfilled a technical goal,
that of providing a finitary means of dealing with arbitrary strict AU-functors between
certain finitely presented AUs.

Many of the technical details are open to change.
It would be great, for instance, if a simpler characterization of AUs could be found.
Nonetheless, I believe the broad approach of sketches,
with equivalence extensions and object equalities,
has the potential for a more enduring usefulness.
One piece of necessary further work is to clarify the connection with the type theory for AUs
as set out in~\cite{Maietti:JAUviaTT}.

The present construction is surely a necessary technical first step
in pursuing the programme set out in~\cite{TopCat},
with its goal of providing a uniform, base-independent technique for proving
results about toposes as generalized spaces,
and even of providing a satisfactory alternative account of generalized spaces.

Over the years, experience with using geometric logic has shown that much of the argument
can be conducted with the ``arithmetic'' AU constraints.
An immediate direction of investigation is to attempt to express them within the
finitary formalism developed in the present paper.

Another pressing need is for a coherent account of the ``geometricity'' properties
of point-free hyperspaces and related constructions.
Current accounts such as that of~\cite{PPExp} prove that the constructions
are preserved up to isomorphism by pullback of bundles,
but do not express any coherence properties of those isomorphisms.
It is to be hoped that that will become clearer in the arithmetic account
when bundles are understood as extensions.
\section{Acknowledgements}
I am grateful to the organizers of the 5th Workshop on Formal Topology,
held at the Institute Mittag-Leffler, Stockholm, on 8-10 June 2015,
for the opportunity to outline the ideas of this paper there.

\bibliographystyle{amsalpha}
\bibliography{MyBiblio}
\end{document}